\documentclass[leqno,twoside,11pt]{amsart}

\usepackage{latexsym,esint,url, amssymb, yfonts}
\usepackage{graphicx,color}

\theoremstyle{plain}
\def\endproof{\hspace*{\fill}\mbox{\ \rule{.1in}{.1in}}\medskip }

\newtheorem{theorem}{Theorem}[section]
\newtheorem{corollary}[theorem]{Corollary}
\newtheorem{lemma}[theorem]{Lemma}
\newtheorem{proposition}[theorem]{Proposition}

\theoremstyle{definition}

\newtheorem{remark}[theorem]{Remark}
\newcommand{\R}{\mathbb{R}} 

\newcommand{\bee}{\begin{equation}}
\newcommand{\eee}{\end{equation}} 
\newcommand{\bees}{\begin{equation*}}
\newcommand{\eees}{\end{equation*}} 
\newcommand{\ds}{\displaystyle} 

\numberwithin{equation}{section}
\numberwithin{figure}{section}

\def\bees{\begin{equation*}}
\def\eees{\end{equation*}}
\def\bee{\begin{equation}}
\def\eee{\end{equation}}
 
\numberwithin{equation}{section}
\numberwithin{figure}{section}

\def\ds{\displaystyle}
\def\Det{{\mathcal{D}et}\,}
\def\cc{{\rm curl}~ {\rm curl}\,} 
\def \sym{{\rm sym}\,} 
\def\R{{\mathbb R}}

\definecolor{Green}{rgb}{0, 0.65,0}

\begin{document}

\title[Convex integration for the Monge-Amp\`ere equation]
{Convex integration for the Monge-Amp\`ere equation \\ in two dimensions} 
\author{Marta Lewicka and Mohammad Reza Pakzad}
\address{Marta Lewicka and Mohammad Reza Pakzad, University of Pittsburgh, Department of Mathematics, 
139 University Place, Pittsburgh, PA 15260}
\email{lewicka@pitt.edu, pakzad@pitt.edu} 
\subjclass[2000]{}


\begin{abstract} 
This paper concerns the questions of flexibility and rigidity
of solutions  to the Monge-Amp\`ere equation which arises as a
natural geometrical constraint in prestrained nonlinear elasticity. In
particular, we focus on anomalous i.e. ``flexible''
weak solutions that  can be constructed through methods of convex
integration \`a la Nash \& Kuiper and establish the related
$h$-principle for the Monge-Amp\`ere equation in two dimensions.
\end{abstract}

\maketitle
\tableofcontents

\section{Introduction.}\label{intro}

In this paper we study the $\mathcal{C}^{1,\alpha}$ solutions to
the Monge-Amp\`ere equation in two dimensions: 
\bee \label{MA}
\Det \nabla^2 v  := -\frac 12 \cc  (\nabla v \otimes \nabla v) =  f
\qquad \mbox{in } \Omega\subset\R^2.
\eee 
Our results concern the dichotomy of ``rigidity vs. flexibility'', in
the spirit of the analogous results and techniques appearing in the
contexts of: the low co-dimension isometric immersion problem \cite{Nash2, kuiper,
  bori1, bori2, CDS}, and the Onsager's conjecture for
Euler equations \cite{Sz, DS0, DS, CoTi,  eyink}. 

In the first, main part of the paper we show that below the
regularity threshold $\alpha<1/7$,  the very weak
$\mathcal{C}^{1,\alpha}(\bar\Omega)$ solutions to
(\ref{MA}) as defined below, are dense in the set
of all continuous functions (see Theorems \ref{weakMA} and
\ref{w2oi-hld}). These flexibility statements are a consequence of 
the convex integration $h$-principle, that is a method 
proposed in \cite{gromov} for solving certain partial differential   
relations and that turns out to be applicable to our setting of the Monge-Amp\`ere
equation as well. Here, we directly adapt the
iteration method of Nash and Kuiper \cite{Nash2, kuiper}, in order to construct the
oscillatory solutions to (\ref{MA}). \footnote{We remark that the
recent work of De Lellis, Inaunen and Szekelyhidi \cite{1/5} showed
that the flexibility exponent $\frac{1}{7}$ can be improved to
$\frac{1}{5}$ in the case of the isometric immersion problem in $2$
dimensions. We expect similar improvement to be possible also in the
present case of equation (\ref{MA}); this will be investigated in
our future work.}

In the second part of the paper we prove that the same class of very
weak solutions fails the above flexibility in the regularity regime
$\alpha>2/3$. Our results are parallel with those concerning  
isometric immersions \cite{bori1, CDS, Pak}, Euler
equations \cite{CoTi, eyink}, Perona-Malik equation \cite{PM1, PM2},
the active scalar equation \cite{IV},
and should also be compared with results on the regularity of Sobolev solutions to the
Monge-Amp\`ere equation \cite{Pak, Sve, LMP, JP} whose study is important
in the context of nonlinear elasticity and with the rigidity results
for the Monge-Amp\`ere functions \cite{Je1, J2010}.

\subsection{The weak determinant Hessian.}
Let $\Omega\subset \R^2$ be an open set.  
Given a function $v\in W^{1,2}_{loc}(\Omega)$, we define its very weak Hessian
(denoted by ${\mathcal H^\ast_2}$ in \cite{Iwa, FM}) as: 
$$ \Det \nabla^2v = -\frac 12 \cc  (\nabla v \otimes \nabla v), $$
understood in the sense of distributions.
A straightforward approximation argument shows that if $v\in W^{2,2}_{loc}$ then $ L^1_{loc} (\Omega) \ni \Det
\nabla^2 v = \det \nabla^2 v$ a.e. in $\Omega$, 
where $\nabla^2 v$ stands for the Hessian matrix field of $v$.  
We also remark that this notion of the very weak Hessian is distinct from the distributional Hessian ${\rm Det}
\nabla^2 v= {\rm Det}  \nabla (\nabla v)$ (denoted by ${\mathcal H}u$
in \cite{Iwa, FM}), that is defined through the distributional determinant ${\rm Det}$:
$$ {\rm Det} \nabla\psi = -\mbox{div } \big(\psi_2\nabla^\perp\psi_1\big)
= \partial_2 (\psi_2 \partial_1 \psi_1)  -\partial_1
(\psi_2 \partial_2 \psi_1) \quad \mbox{ for}   \quad \psi=(\psi_1,
\psi_2)\in W^{1, 4/3}(\Omega,\R^2).$$ 
Contrary to the distributional Hessian, the
very weak Hessian is not continuous with respect to the weak topology. Indeed,
an example of a sequence $v_n \in W^{1,2}(\Omega)$  is constructed in
\cite{Iwa}, where $\Det \nabla^2 v = -1$ while $v_n$ converges weakly to $0$.
One consequence of the proof of our Theorem \ref{weakMA} below is that
$\Det \nabla^2$ is actually weakly discontinuous everywhere in
$W^{1,2}(\Omega)$ (see Corollary \ref{ss}). 

\medskip
 
Here is our first main result:

\begin{theorem}\label{weakMA} 
Let $f \in L^{7/6} (\Omega)$ on an open, bounded, simply connected
$\Omega\subset\mathbb{R}^2$.  Fix an exponent:
$$\alpha<\frac{1}{7}.$$ 
Then the set of $\mathcal{\mathcal{C}}^{1,\alpha}(\bar\Omega)$ solutions to
(\ref{MA}) is dense in the space $\mathcal{\mathcal{C}}^0(\bar\Omega)$. 
More precisely, for every $v_0\in \mathcal{\mathcal{C}}^0(\bar\Omega)$ there exists a sequence
$v_n\in\mathcal{\mathcal{C}}^{1,\alpha}(\bar\Omega)$, converging uniformly to $v_0$ and satisfying:
\bee \label{MA2}
\Det \nabla^2 v_n  = f \quad \mbox{ in } \Omega.
\eee 
When $f\in L^p(\Omega)$ and $p\in (1,\frac{7}{6})$, the same result is true
for any $\alpha<1- \frac{1}{p}$.
\end{theorem} 
 
In order to better understand Theorem \ref{weakMA}, we point out a
connection between the solutions to \eqref{MA} and the  
isometric immersions of Riemannian metrics, motivated 
by a study of nonlinear elastic plates. 
Since on a simply connected domain $\Omega$, the kernel of the differential operator $\cc$ consists 
of the fields of the form $\sym \nabla w$, a solution to \eqref{MA} with the vanishing right
hand side $f\equiv 0$ can be characterized by the criterion: 
\begin{equation}\label{2nd}
\exists w:\Omega \to \R^2 \qquad \frac 12 \nabla v \otimes \nabla v +
\sym \nabla w =0 \quad \mbox{in} \,\, \Omega.
\end{equation}  
The equation in (\ref{2nd}) can be seen as an equivalent condition for
the following $1$-parameter family of
deformations, given through the out-of-plane displacement $v$ and
the in-plane displacement $w$
(albeit with different orders of magnitude $\varepsilon$ and  $\varepsilon^2$):
$$ \phi_\varepsilon= {id} + \varepsilon v e_3 + \varepsilon^2 w: \Omega\to\R^3 $$ 
to form a  $2$nd order infinitesimal isometry (bending), i.e. to induce the change of
 metric on the plate $\Omega$ whose $2$nd order terms in $\varepsilon$ disappear:
$$(\nabla \phi_\varepsilon)^T \nabla\phi_\varepsilon - \mbox{Id}_2 = o(\varepsilon^2).$$ 

In this context, we take the cue about Theorem
\ref{weakMA} from the celebrated work of Nash and Kuiper \cite{Nash2,
  kuiper}, where they show the density of co-dimension one
$\mathcal{C}^1$ isometric immersions of Riemannian manifolds in the
set of short mappings. Since we are are now dealing with the $2$nd order
infinitesimal isometries rather than the exact isometries, the
classical metric pull-back equation:  
$$ y^\ast g_e = h,   $$ 
for a mapping $y$ from $(\Omega,h)$ into $\R^3$ equipped with 
the standard Euclidean metric $g_e$,  is replaced by the 
compatibility equation of the tensor $T(v,w)= \frac 12 \nabla v \otimes
\nabla v + \sym \nabla w$ with a matrix field $A_0$ that satisfies: $ - \cc A_0 =f$:
\begin{equation}\label{1.4} 
T(v,w) = A_0  
\end{equation}
Note that there are many potential choices for $A_0$, for example one
may take $A_0(x) = \lambda(x) \mbox{Id}_2$ with $\Delta\lambda = -f$ in $\Omega$.
Again, equation (\ref{1.4}) states precisely that the metric
$(\nabla\phi_\varepsilon)^T\nabla\phi_\varepsilon$ agrees with the
given metric $h=\mbox{Id}_2 + 2\varepsilon^2 A_0$ on $\Omega$, up to
terms of order $\varepsilon^2$. The Gauss curvature $\kappa$ of the metric $h$
satisfies:
$$\kappa(h) = \kappa(\mbox{Id}_2 + 2\varepsilon^2 A_0) =
-\varepsilon^2 \cc A_0 + o(\varepsilon^2),$$
while $\kappa((\nabla\phi_\varepsilon)^T\nabla\phi_\varepsilon) =
-\varepsilon^2\cc\big(\frac{1}{2}\nabla v\otimes\nabla v + \sym w\big) +
o(\varepsilon^2)$, so the problem (\ref{MA}) can also be interpreted as
seeking for all appropriately regular out-of-plane displacements $v$
that can be matched, by a higher order in-plane displacement perturbation $w$, to
achieve the prescribed Gauss curvature $f$ of $\Omega$, at its highest
order term.

In this paper, similarly as in the isometric immersion case, we show that solutions to
(\ref{1.4}) are ample. We design a scheme inspired by the work
of Nash and Kuiper, which pushes a ``short infinitesimal isometry'',
i.e. a couple $(v_0, w_0)$ such that $T(v_0, w_0) <A_0$, 
towards an exact  solution to (\ref{1.4}) in successive
small steps.   Note that both $y^\ast g_e = (\nabla y)^T \nabla y$ and
the term $\nabla v \otimes \nabla v$ in $T(v,w)$
have a quadratic structure, which is crucial  in the
analysis of \cite{Nash2, kuiper} and also of this paper. Here, not only the presence of the linear term
$\sym \nabla w$ in $T(u,w)$ does not destroy the adaptation of the
Nash-Kuiper scheme, but it actually allows for this construction to work.

\subsection{Convex integration for the Monge-Amp\`ere equation in two dimensions.}
 
As we will see in section \ref{C10}, Theorem \ref{weakMA} follows
easily from the statement of our next main result:
 
\begin{theorem}\label {w2oi-hld}
Let $\Omega\subset\mathbb{R}^2$ be an open and bounded domain. Let
$v_0\in\mathcal{C}^1(\bar\Omega)$, $w_0\in\mathcal{C}^1(\bar\Omega,\mathbb{R}^2)$ and $A_0 \in
\mathcal{C}^{0,\beta} (\bar \Omega,  \R^{2\times 2}_{sym})$, for some $\beta\in (0,1)$, be such that: 
\bee\label{wA-inequ-holder}
\exists c_0>0 \qquad A_0 - \big(\frac 12 \nabla v_0
\otimes \nabla v_0 + \sym \nabla w_0\big)  > c_0 {\rm Id}_2 \qquad \mbox{in } \bar\Omega.
\eee 
Then, for every exponent $\alpha$ in the range:
$$0<\alpha < \min\Big\{\frac{1}{7}, \frac{\beta}{2}\Big\},$$ 
there exist sequences $v_n \in \mathcal{C}^{1,\alpha}(\bar \Omega)$ 
and $w_n \in \mathcal{C}^{1,\alpha} (\bar \Omega,\R^2)$ which converge
uniformly to $v_0$ and $w_0$, respectively, and which satisfy:
\bee\label{2oi-holder} 
A_0 = \frac 12 \nabla v_n \otimes \nabla v_n + \sym \nabla w_n \qquad
\mbox{in } \bar\Omega.
\eee
\end{theorem}

The above result is the Monge-Amp\`ere analogue of  {\cite[Theorem
  1]{CDS}}, where the authors improved on the Nash-Kuiper
method to obtain higher regularity within the flexibility regime.  
In our paper, we adapt the same methods to our problem. 

Convex integration was originally developed by Gromov \cite{gromov} to deal
with finding weak solutions of a differential inclusion $Lu(x)\in
K$  in $\Omega$, by  investigating certain classes of sub-solutions,  
e.g. functions $u$ that satisfy $Lu(x)\in \mbox{conv } K$ where the
original constraint set $K$ is replaced by its
appropriate convex hull $\mbox{conv } K$. Under specific circumstances, it leads to
establishing the density of very weak
solutions, satisfying $Lu \in L^\infty(\Omega)$, in the set of sub-solutions, and in case
the constraint set is a continuum the regularity might be improved to $Lu \in \mathcal{C}^0(\Omega)$.  

Recently, these methods were applied in the context of fluid dynamics
and yielded many  interesting results for  the Euler equations. In
\cite{DS0},  De Lellis and Sz\'ekelyhidi proved existence of weak solutions with bounded
velocity and pressure, their
non-uniqueness and the existence of energy-decreasing solutions. In \cite{DS}, using
iteration methods \`a la Nash-Kuiper, the same authors proved existence of
continuous periodic solutions of 
the $3$-dimensional incompressible Euler equations, which dissipate the total kinetic
energy.  These results are to be contrasted with \cite{CoTi,
  eyink}, where it was shown that $\mathcal{C}^{0,\alpha}$  
solutions of the Euler equations are energy conservative if
$\alpha>1/3$. There have been several improvements of 
\cite{DS0, DS} since, towards a possible proof of the Onsager's
conjecture which puts the H\"older regularity threshold for the energy conservation of
the weak solutions to the Euler equations at $\mathcal{C}^{0,1/3}$
\cite{Is1, Is2, B1, B2, B3, C2}. The stationary incompressible Euler equation has 
been studied in \cite{C2} where the existence of
bounded anomalous solutions have been proved. The authors indicate
that in $2$ dimensions, 
the relaxation set corresponding to the appropriate subsolutions is
smaller than in the case of the evolutionary equations. In this
context, we noticed a connection between our reformulation of the Monge-Amp\`ere equation
and the steady state Euler equation, which lead to our modest  Corollary \ref{corfluid}.

In this paper we use a direct iteration method to construct exact solutions of  (\ref{MA}). The
re-casting of the statement and the proof in the language of convex
integration  might shed more light on the structure of the
Monge-Amp\`ere equation, but it would not improve
the results and therefore we do not address this task. We note, however, that constructing Lipschitz
continuous piecewise affine approximating solutions  to
(\ref{2oi-holder})  for $A_0 \equiv 0$ is quite straightforward and
could be used to prove a convex integration density result via
the Baire category method as was done in \cite{DS0} for the Euler
equations (see also Figure \ref {figpar} and the corresponding explanation).

\subsection{Rigidity versus flexibility.}

The flexibility results obtained in view of the $h$-principle are usually coupled with the rigidity
results for more regular solutions. Rigidity of isometric
immersions of elliptic metrics for $\mathcal{C}^{1,\alpha}$ isometries
\cite{bori1, DS0} with $\alpha>2/3$, or the energy conservation of weak solutions of the
Euler equations for $\mathcal{C}^{0,\alpha}$ solutions
with $\alpha>1/3$, are results of this type.
For the Monge-Amp\`ere equations, we recall two recent statements regarding
solutions with Sobolev regularity:
following the well known unpublished work by \v{S}ver\'ak \cite{Sve}, we proved in
\cite{LMP} that if $v\in W^{2,2}(\Omega)$ is a 
solution to (\ref{MA}) with $f\in L^1(\Omega)$ and $f\geq c>0$ in
$\Omega$, then in fact $v$ must be $\mathcal{C}^1$ and globally convex (or concave). 
On the other hand, if $f=0$ then
\cite{Pak} likewise $v\in\mathcal{C}^1(\Omega)$ and $v$ must be
developable (see also \cite{Je1, J2010, JP}). A clear  {statement
  of rigidity} is still lacking for the general $f$, as is the case for isometric immersions, where rigidity
results are usually formulated only for elliptic \cite{CDS} or Euclidean metrics \cite{Pak, LP, JP}.  

In this paper, we prove the rigidity properties of solutions to
(\ref{MA}) in the H\"older regularity context when $f\equiv 0$. Namely, we prove:

\begin{theorem}\label{rig1} 
Let $\Omega\subset\R^2$ be an open, bounded domain and let:
$$\frac{2}{3}<\alpha<1.$$
If $v\in\mathcal{C}^{1,\alpha}(\bar\Omega)$ is a solution to $\Det
\nabla^2 v = 0$ in $\bar\Omega$, then $v$ must be developable. More
precisely, for all $x\in\Omega$ either $v$ is affine in a
neighbourhood of $x$, or there exists a segment $l_x$ joining
$\partial\Omega$ on its both ends, such that $\nabla v$ is constant on $l_x$.
\end{theorem}

We also announce the following parallel rigidity result  $f\ge c>0$, that will be the
subject of the forthcoming paper \cite{LPconvex}:
\begin{theorem}\label{rig2} 
Let $\Omega\subset\R^2$ be an open, bounded domain and let:
$$\frac{2}{3}<\alpha<1.$$
If $v\in\mathcal{C}^{1,\alpha}(\bar\Omega)$ is a solution to $\Det
\nabla^2 v = f$ in $\bar\Omega$, where $f$ is a positive Dini
continuous, then $v$ is convex. In fact, it is also an
Alexandrov solution to $\det\nabla^2v = f$ in $\Omega$. 
\end{theorem}

In proving Theorem \ref{rig1}, we use  a commutator estimate for
deriving  a degree formula in Proposition \ref{deg}.   
Similar commutator estimates are used in \cite{CoTi} for the Euler equations and in  
\cite{CDS} for the isometric immersion problem;
this is not surprising, since the presence of a quadratic term plays a
major role in  all three cases, allowing for the efficiency  
of the convex integration and iteration methods. Let us also mention that it is still an open problem which
  value of $\alpha$ is the critical value for the rigidity-flexibility dichotomy, and it
  is conjectured to be $1/3, 1/2$ or $2/3$.

\subsection{Notation.}
By $\mathbb{R}^{2\times 2}_{sym}$ we denote the space of symmetric
$2\times 2$ matrices, and by  $\mathbb{R}^{2\times 2}_{sym, >}$ we
denote the cone of symmetric, positive definite $2\times 2$ matrices.
The space of H\"older continuous functions
$\mathcal{C}^{k,\alpha}(\bar\Omega)$ consists of restrictions of
functions $f\in \mathcal{C}^{k,\alpha}(\mathbb{R}^2)$ to $\Omega\subset\R^2$.
Then, the $\mathcal{C}^k(\bar\Omega)$ norm of such restriction is
denoted by $\|f\|_k$, while its H\"older norm $\mathcal{C}^{k, \alpha}(\bar\Omega)$ is $\|f\|_{k,\alpha}$.
By $C>0$ we denote a universal constant which is independent of all
parameters, unless indicated otherwise.

\subsection {Acknowledgments.}
The authors would like to thank Camillo De Lellis for discussions about this problem. 
This project is based upon work supported by, among others, the National Science
Foundation.
M.L. was partially supported by the NSF grants DMS-0846996
and DMS-1406730. M.R.P. was partially supported by the NSF grant
DMS-1210258. A part of this work was completed while the authors
visited the Forschungsinstitut f\"ur Mathematik at ETH 
(Zurich, Switzerland). The institute's hospitality is gratefully acknowledged.

\section{The $\mathcal{C}^1$ approximations - preliminary results.}\label{C1}

In this and the next section we prove a weaker version of the result
in Theorem \ref{w2oi-hld}. Namely:

\begin{theorem}\label {weak2oi}
Let $\Omega\subset\mathbb{R}^2$ be an open and bounded domain. Let $v_0
\in \mathcal{C}^\infty (\bar \Omega)$, $w_0\in \mathcal{C}^\infty(\bar \Omega, \R^2)$ and
$A_0\in \mathcal{C}^\infty (\bar \Omega,  \R^{2\times 2}_{sym})$ be such that:
\bee\label{wA-inequ}
\exists c_0>0 \qquad A_0 - \big(\frac 12 \nabla v_0 \otimes \nabla v_0
+ \sym \nabla w_0\big) > c_0
{\rm Id}_2 \qquad \mbox{in } \bar\Omega.
\eee 
Then there exist sequences $v_n \in \mathcal{C}^1(\bar
\Omega)$ and $w_n \in \mathcal{C}^1(\bar \Omega,   \R^2)$ which converge uniformly
to $v_0$ and $w_0$ respectively, and which satisfy:
\bee\label{2oi} 
A_0 = \frac 12 \nabla v_n \otimes \nabla v_n + \sym \nabla w_n  \qquad
\mbox{in } \bar\Omega.
\eee
\end{theorem}

We start with a series of preliminary lemmas whose details we provide for the
sake of completeness.
The first lemma is an observation in convex integration, pertaining to
solving an appropriate differential inclusion to be used for
constructing the $1$-dimensional oscillatory perturbations in $v_n$
and $w_n$.  As always, $C>0$ is a universal constant, independent of all
parameters, in particular independent of the function $a$ below.
 
\begin{lemma}\label{convex} 
Let $a\in \mathcal{C}^\infty(\bar \Omega)$ be a nonnegative function
on an open and bounded set $\Omega\subset\mathbb{R}^2$.  There
exists a smooth $1$-periodic field $\Gamma = (\Gamma_1, \Gamma_2) \in 
\mathcal{C}^\infty(\bar \Omega \times \R, \R^2)$ such that  the
following holds  for all $(x,t)\in \bar \Omega\times \R$: 
\begin{equation}\label{jed}
\begin{split}
& \Gamma  (x, t+1)   = \Gamma(x,t),  \\
& \frac 12 |\partial_t \Gamma_1(x, t)|^2 + \partial_t \Gamma_2(x, t)  = a(x)^2,   
\end{split} 
\end{equation}
together with the uniform bounds:
\begin{equation}\label{dwa}
\begin{split}
& |\Gamma_1(x,t)| + |\partial_t \Gamma_1 (x,t)|   \le  Ca(x), \qquad
|\Gamma_2(x,t)| + |\partial_t \Gamma_2 (x,t)|   \le  Ca(x)^2, \\
& |\nabla_x \Gamma_1 (x, t)| \le C |\nabla a(x)|, \qquad  |\nabla_x \Gamma_2 (x, t)| \le C |a(x)| |\nabla a(x)|.
\end{split} 
\end{equation}
\end{lemma} 
\begin{proof}
Firstly, note that there exists a smooth $1$-periodic function $\gamma
\in \mathcal{C}^\infty (\R,  \R^2)$, such that for all $t\in \R$ there holds:
\begin{equation*}
\begin{split}
& \gamma  (t+1)   = \gamma(t),  \qquad \int_0^1 \gamma (t) ~\mbox{d}t =(0,0),\\
& \gamma(t) \in P := \Big \{(s_1,s_2) \in \R^2 ; ~ \frac 12 s_1^2 + s_2 =  1, ~|s_1| \le 2  \Big \}.
\end{split} 
\end{equation*}
Existence of $\gamma$ is a consequence of the fundamental lemma of
convex integration, since the intended average $(0,0)$ lies in the
convex hull of the parabola $P$ (see Figure \ref{figpar}).
Indeed, one can take:
\bees
\gamma(t) = \big(2\cos(2 \pi t), - \cos (4 \pi t)\big) \in P.
\eees 
It is now enough to ensure that $\partial_t\Gamma_1 = a(x)
\gamma_1(x)$ and $\partial_t\Gamma_2 = a(x)^2\gamma_2(x)$ to obtain
(\ref{jed}). Namely:
\bees
\Gamma_1(x,t) = \frac{a(x)}{\pi} \sin(2\pi t), \quad \Gamma_2(x,t) = - \frac{a(x)^2}{4\pi} \sin {4\pi t}.
\eees
We see directly that the bounds in (\ref{dwa}) hold.
\end{proof} 

\medskip

To compare with the problem of isometric
immersions, note that in that context, a $1$-dimensional convex integration  
lemma is similarly proved in \cite[Figure2,  p. 11] {Sz}, where instead of a parabola, 
the constraint set consists of a full circle.

\begin{figure}[h]
\includegraphics[width=2.5in]{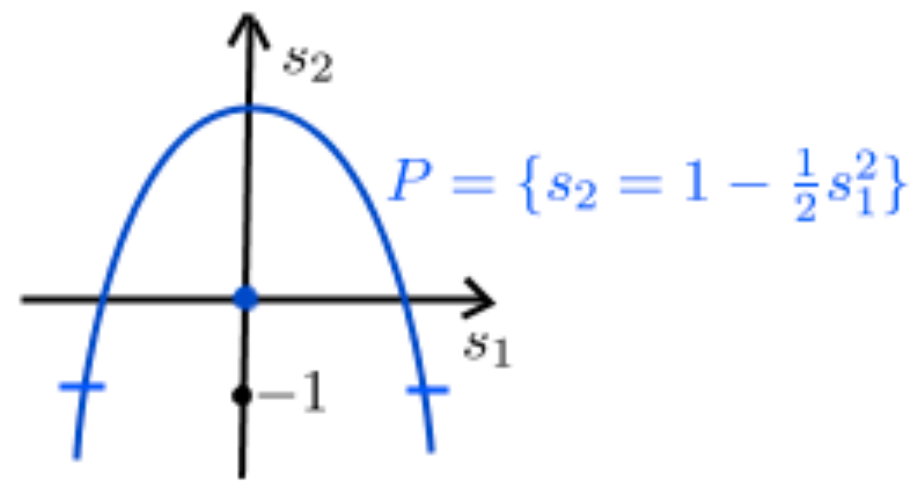}
\vspace{-0.4cm}
\caption{The parabola $P$ in the $1$d convex integration problem of
  Lemma \ref{convex}.} \label{figpar}
\end{figure}

\medskip
 
We will also need a special case of  \cite[Lemma 3]{CDS}  about
decomposition of positive definite symmetric matrices into
rank-one matrices. 

\begin{lemma}\label{decomposeId} 
There exists a sufficiently small constant $r_0>0$ such that the
following holds. For every positive definite symmetric
matrix $G_0\in\mathbb{R}^{2\times 2}_{sym, >}$, there are three unit vectors $\{\xi_k\in\mathbb{R}^3\}_{k=1}^3$ and
three linear functions $\{\Phi_{k} :\R^{2\times 2}_{sym} \to \R\}_{k=1}^3$,
such that:  for any $G\in \R^{2\times 2}_{sym}$ we have
\bee\label{N=3Id}
\forall G\in\mathbb{R}^{2\times 2}_{sym}\qquad G= \sum_{k=1}^3 \Phi_{k}(G) \xi_k \otimes \xi_k,
\eee 
and  that each $\Phi_{k} $ is strictly positive on 
the ball $B(G_0,{r(G_0)})\subset\mathbb{R}^{2\times 2}_{sym}$ with
radius $r(G_0)= \frac{r_0}{|{G_0}^{-1/2}|^2}$.  
\end{lemma} 
\begin{proof}  
{\bf 1.} First, assume that $G_0 = {\rm Id}_2$. Set:
$$\zeta_1= \frac {1}{\sqrt{12}}(2+\sqrt 2, -2+\sqrt 2), \qquad \zeta_2= \frac
{1}{\sqrt{12}}(-2+\sqrt 2, 2+\sqrt 2), \qquad \zeta_3 = \frac {1}{\sqrt{2}} (1,1).$$
In order to check that the following matrices form a basis of the
$3$-dimensional space $\mathbb{R}^{2\times 2}_{sym}$:
$$\zeta_1\otimes\zeta_1 =
\frac{1}{12}\left[\begin{array}{cc} 6+4\sqrt{2} & -2 \\ -2 &
    6-4\sqrt{2} \end{array}\right],\quad \zeta_2\otimes\zeta_2 =
\frac{1}{12}\left[\begin{array}{cc} 6-4\sqrt{2} & -2 \\ -2 &
    6+4\sqrt{2} \end{array}\right],\quad \zeta_3\otimes\zeta_3 =
\frac{1}{2}\left[\begin{array}{cc} 1 & 1 \\ 1 & 1 \end{array}\right],$$
we validate that:
\bees
\det \left ( \ds \frac{1}{12} \left   [  \begin{array}{ccc} 
{6+ 4\sqrt{2}} &  6- 4\sqrt{2} &  6 \\ -2& -2 &  6 \\
6- 4\sqrt{2} &  6+ 4\sqrt{2}&  6 
\end{array} \right ]  \right )  \neq 0.
\eees 
Consequently, there exist linear mappings $\{\Psi_k:\mathbb{R}^{2\times
  2}_{sym}\to\mathbb{R}\}_{k=1}^3$ yielding the unique decomposition:   
\begin{equation}\label{trzy}
\forall G\in\mathbb{R}^{2\times 2}_{sym}\qquad G= \sum_{k=1}^3 \Psi_k(G) \zeta_k \otimes \zeta_k.
\end{equation}
Now, since $ {\rm Id}_2 = \frac 34 \zeta_1 \otimes \zeta_1 + \frac 34 \zeta_2\otimes\zeta_2 
+ \frac 12 \zeta_3 \otimes \zeta_3$, the continuity of each function $\Psi_k$
implies its positivity in a neighborhood of ${\rm Id}_2$ of some
appropriate radius $r_0$.
  
\smallskip

{\bf 2.} For an arbitrary $G_0\in \R^{2\times 2}_{sym, >}$ we set:
\bees 
\forall k=1\ldots3 \qquad \xi_k = \frac{1}{|G_0^{1/2}\zeta_k|}G_0^{1/2}\zeta_k
\quad \mbox{and}\quad \Phi_k(G)=|G_0^{1/2}\zeta_k|^2\Psi_k(G_0^{-1/2}
G  G_0^{-1/2}).
\eees 
Then, in view of (\ref{trzy}) we obtain (\ref{N=3Id}):
\bees
\forall G\in\mathbb{R}^{2\times 2}_{sym}\qquad G =
G_0^{-1/2}\Big(\sum_{k=1}^3 \Psi_k (G_0^{-1/2} G G_0^{-1/2}) \zeta_k
\otimes \zeta_k\Big) G_0^{1/2} = \sum_{k=1}^3 \Phi_k (G) \xi_k \otimes \xi_k,
\eees 
Finally, if $|G- G_0|< r(G_0)$ then $|G_0^{-1/2} G
G_0^{-1/2} - {\rm Id}_2| \leq |G_0^{-1/2}|^2 |G-G_0| < r_0$, and so indeed
$\Phi_k(G)>0$, since $\Psi_k(G_0^{-1/2} G G_0^{-1/2})>0$.
\end{proof} 
  
The above result can be localized in the following manner, similar to \cite[Lemma 3.3]{Sz}:

\begin{lemma}\label{decompose} 
There exists sequences of unit vectors \{$\eta_k\in\mathbb{R}^2\}_{k=1}^\infty$  and 
nonnegative smooth functions $\{\phi_k \in
\mathcal{C}^\infty_c(\mathbb{R}^{2\times 2}_{sym, >})\}_{k=1}^\infty$, such that:
\begin{equation}\label{cztery}
\forall G\in\mathbb{R}^{2\times 2}_{sym, >} \qquad
G = \sum^{\infty}_{k=1}  \phi_k(G)^2 \eta_k \otimes \eta_k
\end{equation}
and that:
\begin{itemize}
\item[(i)] For all $G\in\mathbb{R}^{2\times 2}_{sym, >}$, at most
  $N_0$ terms of the sum in (\ref{cztery}) are nonzero. The constant
  $N_0$ is independent of $G$.
\item[(ii)] For every compact  $K\subset\mathbb{R}^{2\times 2}_{sym, >}$, there
exists a finite set of indices $J(K)\subset \mathbb N$ such that
$\phi_k(G) = 0$ for all $k\not\in J(K)$ and $G\in K$.
\end{itemize}
\end{lemma} 
\begin{proof}
{\bf 1.} Let $r_0$ be as in Lemma \ref{decomposeId} and additionally 
ensure that:
\begin{equation}\label{gru}
r_0<\frac{1}{8}.
\end{equation}
Recall that for each $G\in\mathbb{R}^{2\times 2}_{sym, >}$ we have
denoted $r(G) = \frac{r_0}{|G^{-1/2}|^2}$ and that $B(G, r(G))\subset
\mathbb{R}^{2\times 2}_{sym, >}$. We first construct a locally finite
covering of $\mathbb{R}^{2\times 2}_{sym, >}$ with properties
corresponding to (i) and (ii).

Since the set $\mathbb{R}^{2\times 2}_{sym, >}$ is a cone, we have:
\begin{equation}\label{piec}
\mathbb{R}^{2\times 2}_{sym, >} = \bigcup_{k\in\mathbb{Z}}
2^k\mathcal{C}_0, \quad \mbox{ where } \quad 
\mathcal{C}_0 = \{G\in \mathbb{R}^{2\times 2}_{sym, >};  ~1/2 \le |G|\le 1\}.
\end{equation}
The collection $\{B(G, r(G))\}_{G\in \mathcal{C}_0}$ covers the sector
$\mathcal{C}_0$ by balls that have uniformly bounded radii:
$r(G)\leq r_0\frac{|G|}{\sqrt{2}}\leq r_0$. Hence, by the Besicovitch
covering theorem, it has a countable subcovering ${\mathcal G_0} = \bigcup_{\sigma=1}^{\sigma_0}
\mathcal{G}^\sigma_{0}$, consisting of $\sigma_0\in\mathbb{N}$
countable families $\{\mathcal{G}_0^\sigma\}_{\sigma=1}^{\sigma_0}$
of pairwise disjoint balls.

Note that for all $c>0$ one has:  $r(cG) = c r(G)$ and so: $B(cG,
r(cG)) = cB(G, r(G))$. Consequently, the collections
$\mathcal G_k^{\sigma}= \{ 2^k B; ~ B\in \mathcal G^\sigma_0\}$ each
consist of countably many pairwise disjoint balls, and $\mathcal{G}_k
= \bigcup_{\sigma=1}^{\sigma_0}\mathcal{G}_k^\sigma$ is a covering of
the dilated sector $2^k\mathcal{C}_0$, for every $k\in\mathbb{Z}$. Define:
\begin{equation}\label{szesc}
\forall \sigma=1\ldots \sigma_0 \qquad \mathcal {G}_{even}^\sigma =
\bigcup_{2\mid k} \mathcal{G}_k^{\sigma} \quad \mbox{ and } \quad \mathcal {G}_{odd}^\sigma
= \bigcup_{2\mid (k+1)} \mathcal{G}_k^{\sigma}. 
\end{equation}
Clearly, in view of (\ref{piec}), the $2\sigma_0$ families in
(\ref{szesc}) form a covering of $\mathbb{R}^{2\times 2}_{sym, >}$, namely:
$$\mathcal{G} = \bigcup_{\sigma=1}^{\sigma_0} \mathcal {G}_{even}^\sigma
\cup \bigcup_{\sigma=1}^{\sigma_0} \mathcal {G}_{odd}^\sigma.$$

We now prove that each of the
families in $\mathcal{G}$ consists of pairwise disjoint balls. We
argue by contradiction. Assume that:
\bees
\exists G\in {B}(G_1, r(G_1))\cap {B}(G_2, r(G_2)) \qquad \mbox{for
    some } \quad B(G_1, r(G_1))  \in  \mathcal {G}_{2k_1}^\sigma, \quad 
B(G_2, r(G_2))  \in  \mathcal {G}_{2k_2}^\sigma.
\eees 
Without loss of generality we may take $k_1=0$ and $k_2= k \ge 1$,  so that:
\bees
\frac{1}{2}\le |G_1| \le 1 \quad \mbox{and} \quad  2^{2k-1} \le |G_2| \le 2^{2k}.
\eees 
This yields a contradiction with (\ref{gru}), in view of:
\bees
\begin{split}
2^{2k-1} -1 &  \le |G_2| - |G_1| \le |G_2 - G_1| \le |G_2 - G| +
|G- G_1| \\ &  \le r(G_2) + r(G_1)  = r_0\big(\frac{1}{|G_2^{-1/2}|^2}
+ \frac{1}{|G_1^{-1/2}|^2}\big)
\le \frac{r_0}{\sqrt{2}} (|G_2| + |G_1|) \leq r_0 (2^{2k}+1),
\end{split}  
\eees

\smallskip

{\bf 2.} Note that $\mathcal{G}$ can be assumed locally finite, by
paracompactness. We write: $\mathcal G= \{B_i=B(G_i, r(G_i))\}_{i=1}^\infty$
and let $\{\theta_i\in\mathcal{C}_c^\infty(B_i)\}_{i=1}^\infty$ be a
partition of unity subordinated to $\mathcal{G}$. For each
$i\in\mathbb{N}$, let $\{\xi_{k, G_i}\}_{k=1}^3$ and $\{\Phi_{k,
  G_i}\}_{k=1}^3$ be the unit vectors and the linear functions as in
Lemma \ref{decomposeId}. Then:
\bees
\forall G\in \mathbb{R}^{2\times 2}_{sym, >}\qquad 
G = \sum_{i\in \mathbb N} \theta_i (G) G = \sum_{i\in \mathbb N}
\sum_{k=1}^3  \theta_i(G)\Phi_{k,G_i}(G)\xi_{k,G_i} \otimes
\xi_{k,G_i}, 
\eees 
and we see that (\ref{cztery}) holds by taking:
$$\eta_{i,k} = \xi_{k, G_i} \quad \mbox{ and }\quad \phi_{i,k} =
\big(\theta_i \Phi_{k, G_i}\big).$$
Since $\mbox{supp }\phi_{i,k}\subset \mathcal B_i$ and since each $G$
belongs to at most $2\sigma_0$ balls $B_i$, we see that (i)
holds with $N_0= 6\sigma_0$. On the other hand, condition (ii) follows
by local finiteness of $\mathcal{G}$.
\end{proof}

\section{The $\mathcal{C}^1$ approximations - a proof of Theorem \ref{weak2oi}.}\label{C22}

The first result in the approximating sequence construction is what
corresponds to a \lq step' in Nash and Kuiper's terminology. 
 
\begin{proposition}\label{step}
Let $\Omega\subset \R^2$ be an open and bounded set. Given are:
functions $v \in \mathcal{C}^\infty(\bar \Omega)$ and $w\in
\mathcal{C}^\infty (\bar \Omega, \mathbb{R}^2)$, a nonnegative function
$a \in \mathcal{C}^\infty(\bar \Omega)$, and a unit vector $\eta\in
\R^2$. Then, for every $\lambda>1$ there exist approximations
$\tilde v_\lambda \in \mathcal{C}^\infty(\bar \Omega)$ and $\tilde w_\lambda \in
\mathcal{C}^\infty (\bar \Omega, \mathbb{R}^2)$ satisfying the
following bounds:
\bee\label{stepresult1}
\begin{split}
&\left\|\big(\frac 12 \nabla \tilde v_\lambda \otimes \nabla \tilde v_\lambda +
\sym \nabla \tilde w_\lambda\big)   - \big(\frac 12 \nabla  v \otimes \nabla v
+ \sym \nabla w + a^2 \eta \otimes \eta\big)\right\|_0 \\ &
\qquad\qquad\qquad\qquad \qquad\qquad\qquad\qquad
\leq  \frac C \lambda  
\|a\|_{0}  \big(\|\nabla a\|_{0} + \|\nabla^2 v\|_0  \big) + \frac{C}{\lambda^2} \|\nabla a\|^2_{0}, 
\end{split}
\eee
\bee\label{stepresult2}
\| \tilde v_\lambda -  v\|_0 \le    \frac{C}{\lambda} \|a\|_0 \quad
\mbox{ and } \quad
\| \tilde w_\lambda -  w\|_0  \le \frac{C}{\lambda} \|a\|_0 (\|a\|_0+ \|\nabla v\|_0), 
\eee 
\bee\label{stepresult3}
\begin{split}
& \forall x\in \bar\Omega \qquad 
|\nabla \tilde v_\lambda(x) - \nabla v(x)|  \le C a(x) + \frac C\lambda \|\nabla a\|_0, \\
&|\nabla \tilde w_\lambda(x) - \nabla w(x)|  \le C a(x) (\|a\|_0+
\|\nabla v\|_0) + \frac C{\lambda} \Big(\|a\|_0  (\|\nabla a\|_0 +
\|\nabla^2 v\|_0) + \|\nabla a\|_0 \|\nabla v\|_0\Big). 
\end{split}
\eee 
\end{proposition} 
\begin{proof}
Using the $1$-periodic functions $\Gamma_i$ from Lemma \ref{convex}, we
define $\tilde v_\lambda$ and $\tilde w_\lambda$ as 
$\lambda$-periodic perturbations of $v$, $w$ in the direction $\eta$:
\bee\label{step-ansatz} 
\begin{split}
\tilde v_\lambda  (x) & = v(x) + \frac1{\lambda}  \Gamma_1 (x, \lambda x\cdot\eta) \\
\tilde w_\lambda  (x) & =  w(x) - \frac1{\lambda} \Gamma_1 (x, \lambda x \cdot
\eta) \nabla v(x) + \frac 1\lambda \Gamma_2 (x, \lambda x\cdot \eta) \eta. 
\end{split}
\eee  
The error estimates in \eqref{stepresult2} follow immediately from
(\ref{dwa}). The pointwise error estimates \eqref{stepresult3} follow
from (\ref{dwa}) in view of:
\bees
\begin{split}
\nabla \tilde v_\lambda  (x)  = & ~\nabla v(x) + \frac1{\lambda}
\nabla_x \Gamma_1 (x, \lambda x\cdot \eta) + \partial_t \Gamma_1 (x,
\lambda x \cdot \eta) \eta, \\ 
\nabla \tilde w_\lambda  (x)  =  & ~\nabla w(x) - \frac1{\lambda}\nabla v (x) \otimes
\nabla_x  \Gamma_1 (x, \lambda x \cdot \eta) 
-   \partial_t \Gamma_1 (x, \lambda x\cdot \eta)  \eta \otimes  \nabla v(x)
- \frac1{\lambda} \Gamma_1 (x, \lambda x \cdot \eta) \nabla^2 v(x) \\
& \qquad\quad \quad
+ \frac 1\lambda \eta\otimes \nabla_x
\Gamma_2 (x, \lambda x\cdot \eta) + \partial_t \Gamma_2 (x, \lambda x\cdot \eta)  \eta \otimes \eta.
\end{split}
\eees 
Finally, we compute:
\bees
\begin{split}
\frac 12\nabla \tilde v_\lambda  (x) \otimes &\nabla \tilde v_\lambda  (x)  -
\frac 12 \nabla v(x) \otimes \nabla v(x) \\ & = 
\boxed{\frac{1}{\lambda} \sym\big(\nabla v(x)\otimes \nabla_x  \Gamma_1 (x, \lambda x \cdot
\eta)\big) + \partial_t \Gamma_1 (x, \lambda x\cdot \eta) \sym
\big(\nabla v(x)\otimes\eta\big)} \\ & \quad 
+ \boxed{\boxed{\frac 12 |\partial_t \Gamma_1 (x, \lambda x\cdot \eta)|^2 \eta
\otimes \eta}}  \\ & \quad + \frac{1}{\lambda} \partial_t\Gamma_1(x, \lambda x\cdot\eta)
\sym\big(\eta\otimes \nabla_x\Gamma_1(x,\lambda x\cdot\eta)\big)  
+ \frac{1}{2\lambda^2}  \nabla_x\Gamma_1(x,
\lambda x\cdot\eta) \otimes \nabla_x\Gamma_1(x, \lambda x\cdot\eta), 
\end{split}
\eees 
and: 
\bees
\begin{split}
\sym \nabla \tilde w_\lambda & (x) - \sym \nabla w(x)  \\ & =  
\boxed{-\frac{1}{\lambda} \sym\big(\nabla v(x)\otimes \nabla_x  \Gamma_1 (x, \lambda x \cdot
\eta)\big) - \partial_t \Gamma_1 (x, \lambda x\cdot \eta) \sym
\big(\nabla v(x)\otimes\eta\big)} \\ & \quad 
- \frac{1}{\lambda} \Gamma_1(x, \lambda x\cdot\eta)\nabla^2 v(x)
+ \frac{1}{\lambda} \sym\big(\eta\otimes \nabla_x\Gamma_2(x,\lambda
x\cdot\eta)\big)  \\ & \quad 
+ \boxed{\boxed{\partial_t\Gamma_2 (x, \lambda x\cdot \eta)  \eta \otimes \eta}}.
\end{split}
\eees 
We see that the terms in boxes cancel out, while the terms in double
boxes add up to $a(x)^2\eta\otimes\eta$ in virtue of (\ref{jed}). Consequently:
\bees
\begin{split}
&\Big(\frac 12 \nabla \tilde v_\lambda (x) \otimes \nabla \tilde
v_\lambda (x) + \sym \nabla \tilde w_\lambda (x)\Big)   - \Big(\frac
12 \nabla  v (x)\otimes \nabla v (x) + \sym \nabla w (x) + a(x)^2 \eta
\otimes \eta\Big) \\ & \quad = 
\frac{1}{\lambda} \Big( \partial_t\Gamma_1(x, \lambda x\cdot\eta)
\sym\big(\eta\otimes \nabla_x\Gamma_1(x,\lambda x\cdot\eta)\big)  -
\Gamma_1(x, \lambda x\cdot\eta)\nabla^2 v(x) + \sym\big(\eta\otimes \nabla_x\Gamma_2(x,\lambda
x\cdot\eta)\big)\Big) \\ & \qquad + \frac{1}{2\lambda^2}  \nabla_x\Gamma_1(x,
\lambda x\cdot\eta) \otimes \nabla_x\Gamma_1(x, \lambda x\cdot\eta),
\end{split}
\eees 
which implies \eqref{stepresult1} in view of the bounds in (\ref{dwa}).
\end{proof} 
 
\bigskip

We now complete the \lq stage' in the approximating sequence construction.

\begin{proposition}\label{stage} 
Let $\Omega\subset\mathbb{R}^2$ be an open and bounded domain.
Let $v \in \mathcal{C}^\infty (\bar \Omega)$, $w\in \mathcal{C}^\infty(\bar \Omega, \R^2)$  
and $A\in \mathcal{C}^\infty (\bar \Omega,  \R^{2\times 2}_{sym})$ be
such that  the deficit function $\mathcal{D}$ defined below is positive definite
in $\bar\Omega$:
\bee\label{wA-inequ-stage}
\exists c>0\qquad \mathcal{D}= A- \big(\frac 12 \nabla v \otimes
\nabla v + \sym \nabla w \big) > c \mathrm{Id}_2 \quad \mbox{ in } \bar\Omega.
\eee 
Fix $\varepsilon > 0$. 
Then there exist $\tilde v \in \mathcal{C}^\infty (\bar \Omega)$ and
$\tilde w\in \mathcal{C}^\infty(\bar \Omega, \R^2)$ such that the new
deficit $\tilde{\mathcal{D}}$ is still positive definite, and bounded by
$\varepsilon$ together with the error in the approximations $\tilde
v$, $\tilde w$, namely:
\bee\label{stage4}
\exists \tilde c>0 \qquad \tilde{\mathcal{D}} = A- \big(\frac 12
\nabla \tilde v \otimes  \nabla \tilde v +\sym \nabla \tilde w \big) >
\tilde c \mathrm{Id}_2 \quad \mbox{ in } \bar\Omega,
\eee 
\bee\label{stage12}
\|\tilde{\mathcal{D}}\|_0 < \varepsilon \quad \mbox{ and } \quad
\| \tilde v - v\|_0 +  \|\tilde w - w\|_0 < \varepsilon.
\eee 
Moreover, we have the following uniform gradient error bounds:
\bee\label{stage3}
\|\nabla \tilde v - \nabla v\|_0  \le CN_0^{1/2} \|{\mathcal{D}}\|^{1/2}_0 \quad \mbox{ and } \quad
\|\nabla \tilde w - \nabla w\|_0 \le CN_0 (\|\nabla v\|_0 +
\|\mathcal{D}\|_0^{1/2}) \|{\mathcal{D}}\|^{1/2}_0,   
\eee 
where the constant $N_0\in\mathbb{N}$ is as in Lemma \ref{decompose}.
\end{proposition} 
\begin{proof}  
{\bf 1.} Note that the image $\mathcal{D}(\bar\Omega)$ is a compact
subset of $\mathbb{R}^{2\times 2}_{sym, >}$. By Lemma \ref{decompose}
and  rearranging the indices, if needed, so that
$J(\mathcal{D}(\bar\Omega)) = \{1\ldots N\}$ in (ii), we get:
\bee\label{Bdecompose} 
\forall x\in\bar\Omega \qquad \mathcal{D}(x)= \sum_{k=1}^N
b_k(x)^2\eta_k \otimes \eta_k \quad \mbox{ where } \quad b_k =
\phi_k\circ\mathcal{D}\in\mathcal{C}^\infty(\bar\Omega). 
\eee 
Let now $a_k = (1-\delta)^{1/2}b_k$, with $\delta>0$ so small that:
\begin{equation}\label{1}
\mathcal{D} - \sum_{k=1}^Na_k^2\eta_k\otimes\eta_k =
\delta\mathcal{D}\quad\mbox{ and } \quad \delta\|\mathcal{D}\|_0<\frac{\varepsilon}{2}.
\end{equation}
We set $v_1=v$, $w_1=w$. For $k=1\ldots N$  we inductively define
$v_{k+1}\in\mathcal{C}^\infty(\bar\Omega)$ and $w_{k+1}\in\mathcal{C}^\infty(\bar\Omega, \mathbb{R}^2)$,
by means of Proposition \ref{step} applied to $v_k,
w_k, a_k, \eta_k$ and with $\lambda_k > 1$ sufficiently large
as indicated below. We then finally set $\tilde v = v_{N+1}$ and
$\tilde w=w_{N+1}$.

\smallskip

{\bf 2.} To prove the estimates (\ref{stage4}) - (\ref{stage3}), we start by
observing that since by Lemma \ref{decompose} (i) at most $N_0$ terms
in the expansion (\ref{Bdecompose}) are nonzero, there holds:
\begin{equation}\label{2}
\begin{split}
\sum_{k=1}^N a_k(x) \leq \sum_{k=1}^N b_k(x) & \leq
 N_0^{1/2}\Big(\sum_{k=1}^N b_k(x)^2\Big)^{1/2} = N_0^{1/2}
\big(\mbox{Trace }\mathcal{D}(x)\big)^{1/2} \\ &  \leq N_0^{1/2}
\big(\sqrt{2}~|\mathcal{D}(x)|\big)^{1/2} \leq C N_0^{1/2} \|\mathcal{D}\|_0^{1/2}.
\end{split}
\end{equation}
Further, by (\ref{stepresult1}) and (\ref{1}):
\begin{equation*}
\begin{split}
\tilde{\mathcal{D}} & = \mathcal{D} - \Big(\big(\frac{1}{2}\nabla
\tilde v\otimes\nabla\tilde v + \sym\nabla\tilde w\big) - \big(\frac{1}{2}\nabla
v\otimes\nabla v + \sym\nabla w\big)\Big) \\ &
= \mathcal{D} - \sum_{k=1}^N \Big(\big(\frac{1}{2}\nabla
v_{k+1}\otimes\nabla v_{k+1} + \sym\nabla w_{k+1}\big) - \big(\frac{1}{2}\nabla
v_k\otimes\nabla v_k + \sym\nabla w_k\big)\Big) \\ &
= \Big(\mathcal{D} - \sum_{k=1}^N a_k^2\eta_k\otimes\eta_k\Big)  \\ &
\qquad - 
\sum_{k=1}^N \Big(\big(\frac{1}{2}\nabla
v_{k+1}\otimes\nabla v_{k+1} + \sym\nabla w_{k+1}\big) - \big(\frac{1}{2}\nabla
v_k\otimes\nabla v_k + \sym\nabla w_k +
a_k^2\eta_k\otimes\eta_k\big)\Big) \\ & 
= \delta \mathcal{D} + \sum_{k=1}^N \mathcal{O}\Big(\frac{1}{\lambda_k}
\big(\|a_k\|_0 \|\nabla a_k\|_0 + \|\nabla a_k\|_0^2 + \|a_k\|_0 \|\nabla^2v_k\|_0\big)\Big).
\end{split}
\end{equation*}
Choosing at each step $\lambda_k$ sufficiently large
with respect to the given $a_k$ and the already generated $v_k$, we
may ensure the smallness of the error term in the right hand side
above and hence the positive definiteness of $\tilde{\mathcal{D}}$ in
(\ref{stage4}), because of the uniform positive definiteness of:
$\delta\mathcal{D} > c\delta\mbox{Id}_2$ in $\bar\Omega$. Likewise,
the first inequality in (\ref{stage12}) follows already when the error
is smaller than $\epsilon/2$.

The same reasoning proves the error bounds on $\tilde v-v$ and $\tilde
w - w$ in (\ref{stage12}), in view of (\ref{stepresult2}):
\begin{equation*}
\begin{split}
\tilde v(x) - v(x) & = \sum_{k=1}^N (v_{k+1}(x) - v_k(x)) = \sum_{k=1}^N \mathcal{O}\Big(\frac{1}{\lambda_k}
\|a_k\|_0\Big), \\
\tilde w(x) - w(x) & = \sum_{k=1}^N (w_{k+1}(x) - w_k(x)) = \sum_{k=1}^N \mathcal{O}\Big(\frac{1}{\lambda_k}
\big(\|a_k\|_0^2 +  \|\nabla a_k\|_0 \|\nabla v_k\|_0 \big)\Big).
\end{split}
\end{equation*}

\smallskip

{\bf 3.} To obtain the first error bound in (\ref{stage3}), use (\ref{stepresult3}) and (\ref{2}):
\begin{equation*}
|\nabla\tilde v(x) - \nabla v(x)| \leq \sum_{k=1}^N |\nabla
v_{k+1}(x) -\nabla v_k(x)| \leq C\sum_{k=1}^N a_k(x) + \sum_{k=1}^N \mathcal{O}\Big(\frac{1}{\lambda_k}
\|a_k\|_0^2 \Big)\leq C N_0^{1/2} \|\mathcal{D}\|_0^{1/2},
\end{equation*}
where again, by adjusting $\lambda_k$ at each step, we ensure the
controllability of the error term with respect to the nonnegative
quantity $N_0^{1/2} \|\mathcal{D}\|_0^{1/2}$. Likewise:
\begin{equation*}
\forall k=1\ldots N \qquad |\nabla v_k(x)| \leq |\nabla v(x)| + \sum_{i=1}^{k-1} |\nabla
v_{i+1}(x) -\nabla v_i(x)| \leq \|\nabla v\|_0 + C N_0^{1/2} \|\mathcal{D}\|_0^{1/2},
\end{equation*}
and obviously by (\ref{2}):
\begin{equation*}
a_k(x) \leq \sum_{i=1}^{k-1} a_i(x) \leq C N_0^{1/2} \|\mathcal{D}\|_0^{1/2},
\end{equation*}
which yield by (\ref{2}):
\begin{equation*}
\begin{split}
\sum_{k=1}^N a_k(x) \big(\|a_k\|_0 + \|\nabla v_k\|_0\big) & \leq
C\Big(\|\nabla v\|_0 + N_0^{1/2}\|\mathcal{D}\|_0^{1/2}\Big)
\sum_{k=1}^N a_k(x)  \leq C N_0\big(\|\nabla v\|_0 +
\|\mathcal{D}\|_0^{1/2}\big) \|\mathcal{D}\|_0^{1/2}.
\end{split}
\end{equation*}
Consequently and by (\ref{stepresult3}), there follows the last gradient error bound in
(\ref{stage3}):
\begin{equation*}
\begin{split}
|\nabla\tilde w(x) - \nabla w(x)| & \leq \sum_{k=1}^N |\nabla
w_{k+1}(x) -\nabla w_k(x)| \\ & \leq C\sum_{k=1}^N a_k(x) \big(\|a_k\|_0 +
\|\nabla v_k\|_0\big) \\ & \qquad 
+ \sum_{k=1}^N \mathcal{O}\Big(\frac{1}{\lambda_k}\big (
\|a_k\|_0\|\nabla a_k\|_0  + \|a_k\|_0\|\nabla^2 v_k\|_0 + \|\nabla
a_k\|_0\|\nabla v_k\|_0\big) \Big) \\ & \leq C N_0\big(\|\nabla v\|_0 +
\|\mathcal{D}\|_0^{1/2}\big) \|\mathcal{D}\|_0^{1/2}.
\end{split}
\end{equation*}
This concludes the proof of the stage approximation construction.
\end{proof}  

\bigskip  
   
We now finally give:

\medskip

\noindent {\bf Proof of Theorem \ref{weak2oi}.}  

{\bf 1.} Fix $\varepsilon >0$. It suffices to construct
$v\in\mathcal{C}^1(\bar\Omega)$ and $w\in\mathcal{C}^1(\bar\Omega,
\mathbb{R}^2)$ such that:
\begin{equation}\label{3}
A_0 = \frac{1}{2} \nabla v\otimes\nabla v + \sym\nabla w \qquad
\mbox{in } \bar\Omega
\end{equation}
and:
\begin{equation}\label{4}
\|v-v_0\|_0 + \|w-w_0\|_0 < \varepsilon.
\end{equation}
The exact solution $(v, w)$ of (\ref{3}) will be obtained as the
$\mathcal{C}^1$ limit of sequences of succesive approximations
$\{v_k\in\mathcal{C}^\infty(\bar\Omega),~
w_k\in\mathcal{C}^\infty(\bar\Omega, \mathbb{R}^2)\}_{k=0}^\infty$,
where $v_0$ and $w_0$ are given in the statement of the Theorem and
satisfy (\ref{wA-inequ}), while $v_{k+1}$ and $w_{k+1}$ are defined
inductively by means of Proposition \ref{stage} applied to $v_k, w_k$
and $\varepsilon_k>0$, under the following requirement:
\bee\label{eps} 
\sum_{k=1}^\infty \varepsilon_k < \varepsilon \quad \mbox{ and }  \quad
\sum_{k=1}^\infty \varepsilon_k^{1/2} < 1.
\eee 

In agreement with our notation convention, we introduce the $k$-th
deficit $\mathcal{D}_k$, which is positive definite by (\ref{stage4}):
$$\forall k\geq 0 \qquad \mathcal{D}_k := A_0 - \big(\frac{1}{2}
\nabla v_k\otimes\nabla v_k + \sym\nabla w_k \big)\in
\mathcal{C}^\infty (\bar\Omega, \mathbb{R}^{2\times 2}_{sym, >}).$$ 
By (\ref{stage12}) it follows that:
$$\|v_k - v\|_0 + \|w_k - w\|_0 \leq \sum_{i=0}^{k-1}\|v_{i+1} -
v_i\|_0 + \sum_{i=0}^{k-1}\|w_{i+1} - w_i\|_0 <
\sum_{i=1}^{k-1}\epsilon_i < \sum_{i=1}^\infty\epsilon_i.$$
Thus, $\{v_k\}_{k=0}^\infty$ and $\{w_k\}_{k=0}^\infty$ converge
uniformly in $\bar\Omega$, respectively, to $v$ and $w$ which satisfy
(\ref{4}) in view of (\ref{eps}).

\smallskip

{\bf 2.} We now show that this convergence is in $\mathcal{C}^1$. Indeed, by (\ref{stage12}):
$\|\mathcal{D}_k\|_0<\epsilon_k$, so by (\ref{stage3}):
\bee\label{5}
\|\nabla v_{k+m} - \nabla v_{k}\|_0 \leq \sum_{i=k}^{m-1} \|\nabla
v_{i+1} - \nabla v_i\|_0  \leq C N_0^{1/2} \sum_{i=k}^{m-1} \|\mathcal{D}_i\|_0^{1/2}
\leq C N_0^{1/2} \sum_{i=k}^{m-1} \varepsilon^{1/2}_{i}.
\eee 
In particular, in view of (\ref{eps}) the sequence $\{\|\nabla
v_k\|_0\}_{k=0}^\infty$ is bounded, so we further have:
\bee\label{6}
\begin{split}
\|\nabla w_{k+m} - \nabla w_{k}\|_0 & \leq \sum_{i=k}^{m-1} \|\nabla
w_{i+1} - \nabla w_i\|_0  \\ & \leq C N_0 \sum_{i=k}^{m-1} \big(\|\nabla
v_i\|_0 + \|\mathcal{D}_i\|_0^{1/2}\big) \|\mathcal{D}_i\|_0^{1/2}
\leq \tilde C N_0\sum_{i=k}^{m-1} \varepsilon^{1/2}_{i},
\end{split}
\eee 
where the constant $\tilde C$ is independent of $k$ and $m$. Through
the above assertions (\ref{5}) and (\ref{6}), in view of the second
condition in (\ref{eps}), we conclude that $\{v_k\}_{k=1}^\infty$ and
$\{w_k\}_{k=0}^\infty$ are Cauchy sequences that converge in
$\mathcal{C}^1(\bar\Omega)$ to $v\in \mathcal{C}^1(\bar\Omega)$ and
$w\in \mathcal{C}^1(\bar\Omega, \mathbb{R}^2)$, respectively. Finally:
$$\|A_0 - \big(\frac{1}{2} \nabla v\otimes\nabla v + \sym\nabla w
\big)\|_0 = \lim_{k\to\infty}\|\mathcal{D}_k\|_0 \leq
\lim_{k\to\infty}\varepsilon_k = 0$$
implies (\ref{3}) and completes the proof of Theorem \ref{weak2oi}.
\endproof

\medskip

\begin{remark}\label{remi}
In addition to the uniform convergence postulated in
Theorem \ref{weak2oi}, one also has:
$$\forall n \qquad \|\nabla v_n\|_0\leq \|\nabla v_0\|_0 + CN_0^{1/2}.$$
Using notation as in the proof above and
recalling (\ref{5}) and (\ref{eps}), this bound follows by:
$$\|\nabla v - \nabla v_0\|_0 = \lim_{k\to\infty} \|\nabla v_k -
\nabla v_0\|_0 \leq \lim_{k\to\infty}
\Big(CN_0^{1/2}\sum_{i=0}^{k-1}\epsilon_i^{1/2}\Big) \leq CN_0^{1/2}.$$
\end{remark}

\section{The $\mathcal{C}^{1,\alpha}$ approximations - a proof of
  Theorem \ref{weakMA},  preliminary results and some heuristics
  towards the proof of Theorem \ref{w2oi-hld}.} \label{C10}

Theorem \ref{weakMA} follows easily from Theorem \ref{w2oi-hld}, that will be proved in the next section.

\medskip

\noindent {\bf Proof of Theorem \ref{weakMA}.}

\noindent Since $\mathcal{C}^1(\bar\Omega)$ is dense in
$\mathcal{C}^0(\bar\Omega)$, we may without loss of generality assume
that $v_0\in\mathcal{C}^1(\bar\Omega)$. Set $w_0=0$ and $A_0 =
(\lambda+c)\mbox{Id}\in\mathcal{C}^{0,\beta}(\bar\Omega,\mathbb{R}^{2\times
2}_{sym})$ where $c$ is a constant and $\lambda$ is constructed as follows.

Extend the function $f$ to $f\in L^p(\Omega_\epsilon)$ defined on an open smooth set
$\Omega_\epsilon \supset \bar\Omega$ and solve:
$$ - \Delta\lambda = f \quad \mbox{ in } \Omega_\epsilon, \qquad \lambda = 0 
\quad \mbox{ on } \partial\Omega_\epsilon.$$
Since $\lambda\in W^{2,p}(\Omega_\epsilon)$, then Morrey's Theorem
implies that $\lambda\in\mathcal{C}^{0,\beta}(\bar\Omega)$ for every
$\beta \in (0,1)$ when $p\geq 2$, and for $\beta = 2- \frac 2p$ when $p\in (1, 2)$.
Also, for $c$ large enough, condition (\ref{wA-inequ-holder})
on the positive definiteness of the defect is satisfied. On the other hand:
$$-\mbox{curl }\mbox{curl} A_0 = -\Delta(\lambda + c) = f,$$
so the result follows directly from Theorem \ref{w2oi-hld}, since
$\frac{2-2/p}{2}\geq \frac 17$ is equivalent to $p\geq \frac 76$.
\endproof

\bigskip

Our next simple Corollary concerns the steady-state Euler equations
with the exchanged roles of the given pressure $q$ and the unknown
forcing term $\nabla^\perp g$.

\begin{corollary}\label{corfluid}
Let $\Omega\subset\R^2$ be an open and bounded domain. Let
$q\in\mathcal{C}^{0,\beta}(\bar\Omega)$ for some $\beta\in (0,1)$ and
fix $\epsilon > 0$.
Then for every exponent $\alpha$ in the range: $0 < \alpha < \min\{\frac 17,
  \frac{\beta}{2}\}$, there exist sequences
$\{u_n\in\mathcal{C}^{0,\alpha}(\bar\Omega, \R^2)\}_{n=1}^\infty$ and
$\{g_n\in\mathcal{C}^{0,\alpha}(\bar\Omega)\}_{n=1}^\infty$  solving
in $\Omega$ the following system:
\begin{equation}\label{cpies}
\mathrm{div} (u_n\otimes u_n) - \nabla q = \nabla^\perp g_n, \qquad
\mathrm{div }~ u_n = 0,
\end{equation}
and such that $u_n = \nabla^\perp v_n$ and $g_n = \mathrm{curl }~ w_n$, where each
$v_n\in\mathcal{C}^{1,\alpha}(\bar\Omega)$ and
$w_n\in\mathcal{C}^1(\bar\Omega, \R^2)$, while the sequence
$\{v_n\}_{n=1}^\infty$ is dense in $\mathcal{C}^0(\bar\Omega)$ and
$\|w_n\|_0 < \epsilon$ for every $n\geq 1$.
\end{corollary}

\begin{proof}
As before, since $\mathcal{C}^1(\bar\Omega)$ is dense in
$\mathcal{C}^0(\bar\Omega)$, it is enough to take
$v_0\in\mathcal{C}^1(\bar\Omega)$ and approximate it by a sequence
$\{v_n\in\mathcal{C}^{1,\alpha}(\bar \Omega)\}_{n=1}^\infty$ with the
properties as in the statement of Corollary. Let $w_0 = 0$ and let $c>0$ be a
sufficiently large constant, so that $ (q+c)\mbox{Id}_2 - \nabla
v_0\otimes\nabla v_0$ is strictly positive definite in
$\bar\Omega$. By Theorem \ref{w2oi-hld}, there exists sequences
$v_n\in\mathcal{C}^{1,\alpha}(\bar\Omega)$ and
$w_n\in\mathcal{C}^{1,\alpha}(\bar\Omega, \R^2)$ which converge
uniformly to $v_0$ and $w_0$ and which satisfy:
$$(q+c)\mbox{Id}_2 = \nabla v_n \otimes\nabla v_n + 2\sym\nabla w_n
\quad \mbox{ in } \bar\Omega.$$
Taking the cofactor of both sides in the above matrix identity, we get:
$$(q+c)\mbox{Id}_2 = \nabla^\perp v_n \otimes\nabla^\perp v_n +
2\mbox{cof} \big(\sym\nabla w_n\big).$$
Taking the row-wise divergence, we obtain (\ref{cpies}) with
$u_n=\nabla ^\perp v_n$ and $g_n=\mbox{curl } w_n$, since
$\mbox{div }\mbox{cof }\nabla w_n = 0$, while
$\big(\mbox{div }\mbox{cof }(\nabla w_n)^T\big)^\perp = -\nabla
(\mbox{curl } w_n)$.
\end{proof}

\bigskip

Towards a proof of Theorem \ref{w2oi-hld} we will derive a sequence of
approximation results, and then combine them with Theorem
\ref{weak2oi} in section \ref{nareszcie}.
For completeness, we  first prove a simple, useful:

\begin{lemma}\label{interpol}
Let $\Omega\subset\mathbb{R}^2$ be an open and bounded domain. Given
are functions: $f\in \mathcal{C}^N(\bar \Omega, \R^n)$ and
$\psi\in\mathcal{C}^\infty(\R^n, \R^m)$. Then:
$$\forall k=0\ldots N\qquad \|\psi\circ f \|_k\leq M \|f\|_k,$$ 
where the constant $M>0$ depends on the dimensions $n$, $m$, the
differentiability order $N$, the domain $\Omega$, the norm $\|\psi\|_N$ on the compact set $f(\bar\Omega)$
and the norm $\|f\|_0$, but it does not depend on the higher norms of $f$.
\end{lemma}
\begin{proof}
The statement is obvious for $k=0$. 
Fix $k\in\{1\ldots N\}$ and let $m=(m_1, \cdots, m_k)$ be any $k$-tuple
of nonnegative integers such that $ \sum_{i=1}^k i m_i =k$. Denoting
$|m| = \sum_{i=1}^k m_i $ and using the interpolation inequality \cite{Adams}:
$$ \forall i=1\ldots k \qquad \|f\|_i \le M_0\|f\|^{1-i/k}_{0}
\|f\|^{i/k}_{k},$$
valid with a constant $M_0>0$ depending on $n, N$ and $\Omega$, we get:
$$ \prod_{i=1}^k \|\nabla^i f\|^{m_i}_0 \leq M_0^{|m|} \prod_{i=1}^k \|f\|_0^{m_i-
  im_i/k}\|f\|_k^{im_i/k} =  M_0^{|m|}\|f\|^{|m|-1} _0 \|f\|_k.  $$
with $|m|:= m_1+\cdots + m_j$. 
Calculating the partial derivatives in $\nabla^k (\psi\circ f)$ by the
Fa\`a di Bruno formula, gives hence the desired estimate:
$$ \|\nabla^k (\psi\circ f)\|_0 \leq M \sum_{m} \prod_{i=1}^k
\|\nabla^i f\|^{m_i}_0 \leq M \|f\|_k. $$  
Above, the summation extends over all multiindices $m=(m_1, \cdots, m_k)$ with the
properties listed at the beginning of the proof.
\end{proof}

\medskip

We recall the following estimates which have been proved in \cite{CDS}:

\begin{lemma}\label{stima}
Let $\varphi\in\mathcal{C}_c^\infty(B(0,1),\mathbb{R})$ be a standard
mollifier supported on the ball $B(0,1)\subset\mathbb{R}^n$, that is a
nonnegative, smooth and radially symmetric function such that
$\int_{\mathbb{R}^n}\varphi = 1$. Denote: 
$$\forall l\in (0,1) \qquad \varphi_l (x) =
\frac{1}{l^n}\varphi(\frac{x}{l}).$$
Then, for every $f,g\in\mathcal{C}^0(\mathbb{R}^n)$ there holds:
\begin{eqnarray}
\label{a} \forall k,j\geq 0\quad & & \|f\ast\varphi_l\|_{k+j} \leq \frac{C}{l^k}\|f\|_j\\
\label{b}   \forall k\geq 0 \quad & & \|f\ast\varphi_l -
  f\|_k\leq \frac{C}{l^{k-2}}\|f\|_2\\
\label{c}   \forall \alpha\in (0,1] \quad & & \|f\ast\varphi_l -
  f\|_0\leq {C}{l^\alpha}\|f\|_{0,\alpha}\\
\label{ba}   \forall \alpha\in (0,1] \quad & & \|f\ast\varphi_l\|_1\leq \frac{C}{l^{1-\alpha}}\|f\|_{0,\alpha}\\
\label{d}   \forall k\geq 0 \quad \forall \alpha\in (0,1]\quad & & \|(fg)\ast\varphi_l -
  (f\ast\varphi_l) (g\ast\varphi_l)\|_k\leq \frac{C}{l^{k-2\alpha}}\|f\|_{0,\alpha} \|g\|_{0,\alpha},
\end{eqnarray}
with the uniform constants $C>0$ depending only on the smoothness
exponents $k$, $j$, $\alpha$.
\end{lemma}
\begin{proof}
The estimate (\ref{a}) follows directly from the definition of
convolution. To prove (\ref{b}), note that for every $x\in\mathbb{R}^n$:
\begin{equation*}
\begin{split}
|\nabla^k (f\ast&\varphi_l - f)(x)|  =
\Big|\int_{\mathbb{R}^n} \varphi_l(y) \big(\nabla^k f(x-y) - \nabla^k
f(x)\big)\mbox{d}y\Big| \\ & = \Big|\int_{\mathbb{R}^n} \nabla^k\varphi_l(y) \big(f(x-y) - f(x)\big)\mbox{d}y\Big|
= \frac{1}{l^k} \Big|\int_{\mathbb{R}^n}
\frac{1}{l^n}\nabla^k\varphi(\frac{y}{l}) \big(\nabla f(x)\cdot y +
r_x(y) \big)\mbox{d}y\Big| \\ & 
= \frac{1}{l^k} \Big|\int_{\mathbb{R}^n} \frac{1}{l^n}\nabla^k\varphi(\frac{y}{l}) r_x(y) \mbox{d}y\Big|
\leq \frac{C}{l^k}\sup_{x\in\mathbb{R}^n;~ |y|<l} |r_x(y)|\leq \frac{C}{l^{k-2}} \|f\|_2,
\end{split}
\end{equation*}
where we integrated by parts, discarded the contribution with the
symmetric term $\nabla f(x)\cdot y$ which integrates to $0$, and
estimated the Taylor's formula remainder term: 
$$r_x(y) = f(x-y) - f(x) - \nabla f(x)\cdot y = \|f\|_2 \mathcal{O}(|y|^2).$$
The proof of (\ref{c}) follows similarly by:
\begin{equation*}
\begin{split}
|\nabla^k (f\ast\varphi_l - f)(x)|  & =
\Big|\int_{\mathbb{R}^n} \varphi_l(y) |y|^\alpha \frac{f(x-y) -
  f(x)}{|y|^\alpha}~\mbox{d}y\Big| \\ & \leq C l^\alpha \|f\|_{0,\alpha}
\int_{\mathbb{R}^n} \varphi_l(y) \mbox{d}y\leq C l^\alpha \|f\|_{0,\alpha},
\end{split}
\end{equation*}
while for (\ref{ba}) we write:
\begin{equation*}
\begin{split}
|\nabla (f\ast\varphi_l)(x)| & = \Big|\int_{\mathbb{R}^n}
f(x-y)  \frac{1}{l^{n+1}} \nabla\varphi_l(\frac{y}{l}) ~\mbox{d}y\Big| 
= \frac{1}{l}\Big|\int_{\mathbb{R}^n}
\frac{f(x-y) - f(x)}{|y|^\alpha}  \frac{|y|^\alpha}{l} \frac{1}{l^n} \nabla\varphi_l(\frac{y}{l}) ~\mbox{d}y\Big| 
\\ & \leq C l^{\alpha-1} \|f\|_{0,\alpha}
\int_{\mathbb{R}^n} \frac{1}{l^n} \big|\nabla\varphi_l(\frac{y}{l})\big| ~\mbox{d}y\leq
\frac{C}{l^{1-\alpha}} \|f\|_{0,\alpha}.
\end{split}
\end{equation*}
Finally, for the crucial commutator estimate (\ref{d}) we refer to
\cite[Lemma 1]{CDS}.
\end{proof}

\bigskip

\noindent {\bf A heuristic overview of the next two sections.} 

\noindent Let us attempt to follow the construction in sections
\ref{C1} and \ref{C22}, but with the goal of controlling 
the higher H\"older norms of the iterations, and hence also quantifying the growth of the $\mathcal{C}^2$
norms of  $v, w$. Let  $A\in \mathcal{C}^\infty (\bar \Omega, \R^{2\times 2}_{sym})$  be the
target matrix field and let $v_1\in \mathcal{C}^\infty(\bar \Omega)$,  
$w_1\in \mathcal{C}^\infty (\bar \Omega, \R^2)$ be given at an input
of a \lq stage'. As in Proposition \ref{stage}, we decompose the defect $\mathcal{D}= A-
(\frac 12 \nabla v_1 \otimes \nabla v_1 + \sym \nabla w_1)$ into a
linear combination  $\sum_{k=1}^N a^2_k\eta_k\otimes\eta_k$ of rank-one symmetric matrices
with smooth coefficients given by Lemma \ref{decompose}.  We define:
\begin{equation*}\label{stp-anz-hld0} 
\begin{split}
v_{k+1}   (x) & = v_k(x) + \frac{1}{\lambda} \Gamma_1(x, \lambda x\cdot \eta_k), \\
w_{k+1}   (x) & =  w_k(x) - \frac{1}{\lambda} 
\Gamma_1 (x, \lambda x \cdot \eta_k) \nabla v_k(x) + \frac {1}\lambda
\Gamma_2(x, \lambda x\cdot \eta_k) \eta_k. 
\end{split}
\end{equation*} 
This yields, by applying Lemma \ref{interpol}  to $\psi(x)=x^2$ and $f=a_k$:
\begin{equation*}
\begin{split} 
& \forall m: 0\ldots 3 \qquad \|\nabla^m v_{k+1} - \nabla^m v_k\|_0 \le
C \sum_{i+j=m; ~0\leq i,j\leq m} \|a_k\|_i \lambda ^{j-1},\\
& \forall m: 0\ldots 2 \qquad \|\nabla^m w_{k+1} - \nabla^m w_k\|_0
\le C \sum_{i+j=m;~ 0\leq i,j\leq m}  \|a_k\|_i \lambda ^{j-1}   \\ &
\qquad\qquad\qquad\qquad\qquad \qquad\qquad\qquad \qquad + C
\sum_{i+j+ s=m;~0\leq i,j,s\leq m}  \|a_k\|_i \lambda^{j-1}  \|\nabla^{s+1} v_k\|_{0},
\end{split}
\end{equation*}
On the other hand, applying Lemma \ref{interpol} to $\psi =
\phi_k$ defined in Lemma \ref{decompose} and to $f= \mathcal{D}$, we get: 
$$ \forall k:1\ldots N \qquad   \|a_k\|_2 \le C (\|v_1\|^2_{3} + \|w_1\|_{3} + \|A\|_2).  $$ 

Now, in order to control the $\mathcal{C}^{1,\alpha}$ norm of $v_{N+1}$  
through interpolation, we need  to control the norm $\|v_{N+1}\|_2$, which in
turn depends on $\|a_k\|_2$.  The above estimate shows that at the end of each stage,
the $\mathcal{C}^2$ norm of  $a_k$ is determined by the $\mathcal{C}^3$ norms of the given
$v_1$ and $w_1$ of the previous stage.  Further, the   
$\mathcal{C}^2$ norm of $w_{N+1}$ is only controlled by the $\mathcal{C}^{3}$ norm of $v_0$
and also of all the $a_k$'s. One might hope to control $\|a_k\|_3$  
if the deficit $\mathcal{D}$ is small enough, but the dependence of
$\|w_{N+1}\|_2$ on $\|v_0\|_3$ cannot be easily bypassed.  
Recalling that we need infinitely many stages in the construction,
this implies that a direct estimate cannot be obtained in this manner,
unless we deal with analytic data similar as in \cite{bori2}.
We thus need to modify the previous simplistic approach. 

The appropriate modification is achieved by introducing a mollification before each stage.
Indeed, we note that the loss of derivatives in the
above estimates is accompanied by a similar gain in the powers of
$\lambda$, in a manner that the total order of derivatives, plus
the order of powers needed to control $\|v_{N+1}\|_2$ and
$\|w_{N+1}\|_2$  is constant. If we replace $v_1$ and $w_1$ by their mollifications
on the scale $l \sim \lambda^{-1}$, each derivative loss can be
estimated  by  one power of $\lambda$, and  
$\|v_0\|_2$ and $\|w_0\|_2$ will control $\|v_{N+1}\|_2$ and
$\|w_{N+1}\|_2$. One problem still remains to be taken care of: does
the deficit $\mathcal{D}$ decrease at the end of each stage? As
the calculation below will show, a mollification of order
$\lambda^{-1}$ does not suffice to this end,
and we need to mollify at a larger  scale of $l> \lambda^{-1}$. 

This is indeed how we want proceed. 
In practice, we let the mollification scale to be $l= \delta/M$ and we treat
$\nabla v$ ``like $a$'', controlling its $j$-th norm by $\delta l^{-j}$. 
We then ``sacrifice'' one $l$ in order to gain one $\delta$; instead of $\|\nabla (v\ast
\varphi_l)\|_j \le C \|v\|_1 l^{-j}$, we use $\|\nabla (v\ast
\varphi_l)\|_j \le C (\|v\|_2 l) l^{-j}$, choosing $l$ such that
$l\|v\|_2 < \delta$ and obtaining the desired bound (\ref{stephld2}). 

Finally, note that the loss of $N$ powers of
$\lambda l>1$ in the control of the $\mathcal{C}^2$  
norms at the end of each stage, is the main reason why the described
scheme does not deliver better than $\mathcal{C}^{1,1/7}$ 
estimates, even for the optimal $N=3$ from the decomposition
in Lemma \ref{decomposeId}. 


\section{The $\mathcal{C}^{1,\alpha}$ approximations - a \lq step' and
  a \lq stage'  in a proof of Theorem \ref{w2oi-hld}.}
  
In this section, we develop the approximation technique that will be
used for a proof of Theorem \ref{w2oi-hld} in the next section. 
The first result is a variant of Proposition \ref{step} in which we
accomplish the \lq step' of the Nash-Kuiper construction
with extra estimates on the higher derivatives.
 
\begin{proposition}\label{step-hld}
Let $\Omega\subset \R^2$ be an open, bounded set. Given are functions:
$v \in \mathcal{C}^{3}(\bar\Omega)$, $w\in\mathcal{C}^{2} (\bar\Omega,
\R^2)$, a nonnegative function $a \in \mathcal{C}^{3}(\bar \Omega)$
and a unit vector $\eta\in \R^2$. Let $\delta, l \in (0,1)$ be two
parameter constants such that:
\bee\label{step-a0}
\|a\|_m \le \frac{\delta}{l^m} \quad \forall m=0 \ldots 3, 
\quad \mbox{ and } \quad
\| \nabla v\|_{m}   \le \frac{\delta}{l^m} \quad  \forall m=1,2.
\eee
Then for every $\lambda > 1/l$ there exist approximating functions $\tilde
  v_\lambda\in \mathcal{C}^{3}(\bar \Omega)$ and $\tilde w_\lambda\in
  \mathcal{C}^{2}(\bar \Omega, \R^2)$ satisfying the following bounds,
  with a universal constant $C>0$ independent of all parameters:
\bee \label{stephld2}
\Big\|\big(\frac 12 \nabla \tilde v_\lambda \otimes \nabla
\tilde v_\lambda + \sym \nabla \tilde w_\lambda\big) - \big(\frac 12 \nabla v
\otimes \nabla v + \sym \nabla w + a^2
\eta \otimes \eta\big)\Big\|_0 \leq C \frac{\delta^2}{\lambda l}, 
\eee 
\bee\label{stephld1v}
\|\tilde v_\lambda - v\|_m  \leq C\delta  \lambda^{m-1}  \qquad
\forall m=0 \ldots 3,
\eee
\bee\label{stephld1w}
\|\tilde w_\lambda - w\|_m \leq C \delta  \lambda^{m-1} \big(1+
\|\nabla v\|_0\big) \qquad  \forall m=0\ldots 2.
\eee
\end{proposition}  
\begin{proof} 
We define $\tilde v_\lambda$, $\tilde w_\lambda$ as in the proof of
Proposition \ref{step}:
\begin{equation*}
\begin{split}
\tilde v_\lambda(x) & = v(x) + \frac{1}{\lambda}  \Gamma_1(x, \lambda x\cdot \eta), \\
\tilde w_\lambda(x) & =  w(x) - \frac{1}{\lambda} \Gamma_1(x,
\lambda x \cdot \eta) \nabla v(x) + \frac {1}\lambda \Gamma_2(x, \lambda
x\cdot \eta) \eta. 
\end{split}
\end{equation*}
Firstly, (\ref{stephld2}) follows immediately from (\ref{stepresult1})
in view of (\ref{step-a0}), because $\lambda l > 1$:
\begin{equation*}
\frac{1}{\lambda} \|a\|_0 \big(\|\nabla a\|_0 +  \|\nabla^2 v\|_0\big)
+ \frac{1}{\lambda^2} \|\nabla a\|^2_{0} 
\leq  2\frac{\delta}{\lambda} \frac{\delta}{l} + \frac{1}{\lambda^2}
\frac{\delta^2}{l^2} \leq 3 \frac{\delta^2}{\lambda l}.
\end{equation*}
To check (\ref{stephld1v}), we compute directly as in Lemma \ref{convex}:
\begin{equation*}
\nabla^m(\tilde v_\lambda - v)\|_0  \leq  \frac{C}{\lambda}
\|\nabla^m\Gamma_1(x, \lambda x\cdot \eta)\|_0 \leq  \frac{C}{\lambda}
\sum_{i+j=m;~ 0\leq i,j\leq m}\|a\|_j\lambda^j \leq 
\frac{C}{\lambda} \sum_{i=0}^m\frac{\delta}{l^i}\lambda^{m-i} \leq
C\delta \lambda^{m-1}
\end{equation*}
by (\ref{step-a0}) and noting again $\lambda l > 1$.
Similarly:
\bees
\begin{split} 
\|\nabla^m(\tilde w_\lambda - w)\|_0  & \leq   \frac{C}{\lambda} \Big(
\|\nabla^m\Gamma_2(x, \lambda x\cdot \eta)\|_0 +
\|\nabla^m\Gamma_1(x, \lambda x\cdot \eta)\nabla v\|_0 \Big) \\ &
\leq \frac{C}{\lambda} \Big( \sum_{i+j=m, ~ 0\le i,j \le m}  \|a^2\|_i
\lambda ^{j}   + \sum_{i+j+ s=m, ~ 0\le i,j,s\le m}  \|a\|_i
\lambda^{j}  \|\nabla v\|_{s} \Big) \\ & 
\leq \frac{C}{\lambda} \Big( \sum_{i=1}^m\frac{\delta}{l^i}\lambda^{m-i} +
\sum_{0\leq i+s\leq m, ~ 0\le i,s \le m}\frac{\delta}{ l^i}
\lambda^{m-(i+s)}\frac{\delta}{l^s} +  \sum_{i+j=m, ~ 0\le i,j\le m}
\frac{\delta}{l^i} \lambda^{j} \|\nabla v\|_0\Big) \\ &
 \leq \frac{C}{\lambda} \Big(\sum_{i=1}^m \frac{\delta}{l^i}
 \lambda^{m-i} \Big) \big( 1+ 1 + \|\nabla v\|_0\big) 
\leq C \delta \lambda^{m-1} (1 + \|\nabla v\|_0), 
\end{split}   
\eees 
where we applied Lemma \ref{interpol}  to $\psi(x)=x^2$ and $f=a$  
in view of (\ref{step-a0}) yielding $\|a\|_0\leq 1$, so that: 
$\|a^2\|_i \le C \|a\|_i\leq C\delta/l^i$.
This achieves (\ref{stephld1w}) and completes the proof of Proposition.
\end{proof} 

\medskip

We now accomplish the \lq stage' in the H\"older regular approximation construction.

\begin{proposition}\label{stg-hld}
Let $\Omega\subset\mathbb{R}^2$ be an open, bounded domain. Let 
$v\in \mathcal{C}^2(\bar \Omega)$, $w\in \mathcal{C}^2(\bar \Omega, \R^2)$ 
and $A\in \mathcal{C}^{0,\beta}(\bar \Omega, \R^{2\times 2}_{sym})$
for some $\beta\in (0,1)$, be such that the deficit $\mathcal{D}$ is appropriately small:
\begin{equation}\label{assumpkot}
\mathcal{D} = A - \big(\frac{1}{2}\nabla v\otimes\nabla v +
\sym\nabla w\big), \qquad 0 < \|\mathcal{D}\|_0 < \delta_0\ll 1.
\end{equation}
Then, for every two parameter constants $M, \sigma$ satisfying:
\begin{equation}\label{assumpkot2}
M > \max\{\|v\|_2, \|w\|_2, 1\} \quad \mbox{ and } \quad \sigma>1,
\end{equation}
there exists  $ \tilde v \in \mathcal{C}^2(\bar
\Omega)$ and $\tilde w  \in \mathcal{C}^2(\bar \Omega, \R^2)$ such
that  the following error bounds hold for $\tilde v$, $\tilde w$ and
the new deficit $\tilde{\mathcal{D}} = A -
  \big(\frac{1}{2}\nabla \tilde v\otimes\nabla \tilde v + \sym\nabla \tilde w\big)$:
\bee\label{stg-hld1}
\|\tilde{\mathcal{D}}\|_0 \le C \Big(\frac{\|A\|_{0,\beta}}{M^\beta}\|\mathcal{D}\|_0^{\beta/2} +
\frac{1}{\sigma} \|\mathcal{D}\|_0 \Big),    
\eee
\bee\label{stg-hld2v}
\|\tilde v - v\|_1  \leq C \|\mathcal{D}\|_0^{1/2} \quad \mbox{ and }
\quad  \|\tilde w - w\|_1 \leq C (1+ \|\nabla v\|_0) \|\mathcal{D}\|_0^{1/2} , 
\eee
\bee\label{stg-hld3}
\|\tilde v \|_2   \leq C   M \sigma^3  \quad \mbox{ and }
\quad   \|\tilde w\|_2 \le  C (1+ \|\nabla v\|_0) M \sigma^3.
\eee
The constant $C>0$ is universal and independent of all parameters.
\end{proposition}
\begin{proof}
Analogously to \cite[Proposition 4]{CDS}, the proof is split into three parts.

{\bf 1. Mollification.}  Let $\varphi\in\mathcal{C}_c^\infty(B(0,1))$
be the standard mollifier in $2$d, as in Lemma \ref{stima}. Since
$v$, $w$ and $A$ can be extended on the whole $\mathbb{R}^2$, with all
their relevant norms increased at most $C$ times ($C$ depends here on
the curvature of the boundary $\partial\Omega$), we may define:
$$ {\mathfrak v} = v \ast \varphi_l, \qquad  {\mathfrak w} := w
\ast \varphi_l, \qquad   {\mathfrak A}  := A \ast \varphi_l \qquad
\mbox{with } \quad l=\frac{\|\mathcal{D}\|_0^{1/2}}{M} <1. $$
Applying Lemma \ref{stima} and noting (\ref{assumpkot2}), we
immediately get the following uniform error bounds for $ {\mathfrak
  v}$, $ {\mathfrak w}$, $ {\mathfrak A}$ and for the induced deficit 
$ {\mathfrak D} = {\mathfrak A} - \big(\frac{1}{2} \nabla {\mathfrak
  v}\otimes\nabla {\mathfrak v} + \sym\nabla {\mathfrak w}\big)$:
\bee\label{mol1}
\begin{split}
  \| {\mathfrak v}  - v \|_1  + \| {\mathfrak w} - w\|_1  & \leq C l
 (\|v\|_2 + \|w\|_2) \leq C \|\mathcal{D}\|_0^{1/2}, \\
  \| {\mathfrak A} - A\|_0 & \leq C l^\beta \|A\|_{0,\beta}, \\
 \|{\mathfrak D}\|_m & \leq 
\|\mathcal{D}\ast\varphi_l\|_m + 
\| (\nabla v\ast \varphi_l)\otimes ( \nabla v\ast \varphi_l)
- (\nabla v \otimes \nabla v) \ast \varphi_l \|_m \\ & \leq
\frac{C}{l^m} \|\mathcal{D}\| + \frac{C}{l^{m-2}}
\|v\|_2^2 \leq  \frac{C}{l^m} \|\mathcal{D}\|_0 \qquad \forall m=0\ldots 3.
\end{split}
\eee 
In the proof of the last inequality above, we used (\ref{d}) with the H\"older exponent $\alpha=1$.

We note that so far we have simply exchanged the lower regularity
fields $v$, $w$, $A$ with their smooth approximations, at the expense
of the error that,  as we shall see below, is compatible with
the that postulated in (\ref{stg-hld1}) - (\ref{stg-hld3}). The
following estimate, however, reflects the advantage of averaging
through mollification that results in the control of $\mathcal{C}^3$
norm of ${\mathfrak v}$ by the $\mathcal{C}^2$ norm:
\bee\label{step-v}
\forall m=1,2\qquad \| \nabla {\mathfrak v} \|_{m} \leq \|{\mathfrak v}\|_{m+1} \leq 
\frac{C}{l^{m-1}}\|v\|_2 \leq \frac{C}{l^{m-1}}\|\mathcal{D}\|_0^{1/2}, 
\eee 
where again we used Lemma \ref{stima} and (\ref{assumpkot2}). Note
that the scaling bound (\ref{step-v}) is consistent with the second
requirement in (\ref{step-a0}) of Proposition \ref{step-hld}.
We also record the following simple bound:
\bee\label{step-w}
\|\mathfrak w\|_2 \le C \|w\|_2 \le CM.
\eee
 
\smallskip

{\bf 2. Modification and positive definiteness.} 
Contrary to the \lq stage' construction in the  proof of Proposition \ref{stage}, we do not know
whether the original defect $\mathcal{D}$ (and hence the induced
defect ${\mathfrak D}$) is positive definite, so that Lemma
\ref{decompose} could be used. In any case, we need to keep the
number of terms in the decomposition (\ref{Bdecompose}) into rank-one matrices
as small as possible. 

We now further modify ${\mathfrak w}$ in order to use the optimal
decomposition in (\ref{N=3Id}). Let $r_0$ be as in Lemma
\ref{decomposeId} and define:
$${\mathfrak w}'= {\mathfrak w}  - 2\frac{(\|{\mathfrak D}\|_0 +
    \|\mathcal{D}\|_0)}{r_0}id_2, \qquad \quad 
 {\mathfrak D}' = {\mathfrak A} - \big(\frac{1}{2} \nabla {\mathfrak
  v}\otimes\nabla {\mathfrak v} + \sym\nabla {\mathfrak w}'\big). $$
Clearly, by (\ref{mol1}) we get:
\begin{equation}\label{res}
\|{\mathfrak w}' - {\mathfrak w} \|_2 \leq C (\|{\mathfrak D}\|_0 +
    \|\mathcal{D}\|_0)\leq C \|\mathcal{D}\|_0. 
\end{equation}

Note now that:
$$ {\mathfrak D}'(x) = 2\frac{(\|{\mathfrak D}\|_0 +
    \|\mathcal{D}\|_0)}{r_0}\mbox{Id}_2 + {\mathfrak D}(x) = 
2\frac{(\|{\mathfrak D}\|_0 + \|\mathcal{D}\|_0)}{r_0}
\Big(\mbox{Id}_2 + \frac{r_0}{2(\|{\mathfrak D}\|_0 +
    \|\mathcal{D}\|_0)}{\mathfrak D}\Big) \qquad \forall x\in\bar\Omega. $$
By Lemma \ref{decomposeId} we may apply (\ref{N=3Id}) to the scaled defect
$ G= {\rm Id}_2 + \frac{r_0}{2(\|{\mathfrak D}\|_0 +  \|\mathcal{D}\|_0)}{\mathfrak D}$
and arrive at:
\begin{equation}\label{mo}
{\mathfrak D}'(x) = \sum_{k=1}^3 2\frac{(\|{\mathfrak D}\|_0 +
  \|\mathcal{D}\|_0)}{r_0} \Phi_k(G(x)) \xi_k \otimes \xi_k  = 
\sum_{k=1}^3 a_k^2(x) \xi_k \otimes \xi_k 
\qquad \forall x\in\bar\Omega,
\end{equation}
where $\Big\{a_k = \big(2\frac{(\|{\mathfrak D}\|_0 + \|\mathcal{D}\|_0)}{r_0} 
\Phi_{k}\circ G\big)^{1/2}\Big\}_{k=1}^3$ are positive smooth functions on $\bar \Omega$. 
We claim that:   
\bee\label{step-a}
\forall k=1\ldots 3 \quad \forall m=0\ldots 3
\qquad \|a_k\|_m \leq  \frac{C}{l^m}\|\mathcal{D}\|_0^{1/2}.
\eee 
Indeed, for $m=0$ this inequality follows directly by
$\|{\mathfrak D}\|_0 \leq C\|\mathcal{D}\|_0$. For $m=1\ldots 3$ we
use Lemma \ref{interpol} to each $\psi=\Phi_k^{1/2}$ and $f=G$, where
noting that $\|G\|_0\leq C$ and recalling (\ref{mol1}) yields:
\bee\label{step-am}
\begin{split}
\|a_k\|_m & \leq \Big(2\frac{(\|{\mathfrak D}\|_0 +
  \|\mathcal{D}\|_0)}{r_0}\Big)^{1/2} C \|G\|_m \\ & \leq C(\|{\mathfrak D}\|_0 +
  \|\mathcal{D}\|_0)^{1/2}\Big(C + \frac{r_0}{2(\|{\mathfrak D}\|_0 +
  \|\mathcal{D}\|_0)}\|{\mathfrak D}\|_m\Big) \\ & \leq 
C \Big((\|{\mathfrak D}\|_0 +
  \|\mathcal{D}\|_0)^{1/2} + \frac{1}{(\|{\mathfrak D}\|_0 +
  \|\mathcal{D}\|_0)^{1/2}}\frac{1}{l^m}\|\mathcal{D}\|_0\Big) \leq 
C \Big( \|\mathcal{D}\|_0^{1/2} + \frac{1}{l^m}\|\mathcal{D}\|_0^{1/2}\Big)
\end{split}
\eee
and hence achieves (\ref{step-a}). Note that the scaling bound
(\ref{step-a}) is consistent with the first requirement in (\ref{step-a0}) of Proposition \ref{step-hld}.

\smallskip 

{\bf 3. Iterating the one-dimensional oscillations.} 
We set $v_1={\mathfrak v}$, $w_1={\mathfrak w}$ and inductively define
$v_{k+1}\in\mathcal{C}^3(\bar\Omega)$ and $w_{k+1}\in\mathcal{C}^2(\bar\Omega, \mathbb{R}^2)$
for $k=1,2,3$ by means of Proposition \ref{step-hld} applied to $v_k$,
$w_k$, the function $a_k$ and the unit vector $\xi_k$ appearing in
(\ref{mo}), with the parameters:
$$l_k = \frac{l}{\sigma^{k-1}} <1, \qquad \lambda_k =
\frac{1}{l_{k+1}} > \frac{1}{l_k},$$
and with the remaining three parameters:
\begin{equation}\label{ko2}
\delta_3 \geq \delta_2\geq\delta_1 = \max_{m=1,2} \big\{l^m\|\nabla {\mathfrak v}\|_m\big\}
+ \max_{m=0\ldots 3, ~ k=1\ldots 3} \big\{l^m\|a_k\|_m\big\} 
\end{equation}
as indicated below. We then finally set: $\tilde v = v_4$ and $\tilde w = w_4$.

We start by checking that the assumptions of Proposition
\ref{step-hld} are satisfied. Namely, we claim that $\delta_k, l_k\in
(0,1)$ together with:
\begin{equation}\label{2.23}
\|a_k\|_m \leq\frac{\delta_k}{l_k^m} \quad \forall m=0\ldots 3 \qquad
\mbox{ and } \qquad \|\nabla v_k\|_m \leq \frac{\delta_k}{l_k^m} \quad
\forall m=1,2,
\end{equation}
at each iteration step $k=1,2,3$, if only the constant $\delta_0$ in
(\ref{assumpkot}) is appropriately small.

Indeed, $\delta_1\leq C \|\mathcal{D}\|_0^{1/2}$ in view of (\ref{step-v}) and (\ref{step-a}),
so $\delta_1<1$ if only $\delta_0\ll 1$. Further, by the definition
(\ref{ko2}) it follows that: 
$\|a_k\|_m = \frac{1}{l^m} l^m \|a_k\|_m \leq \frac{\delta_1}{l^m}\leq
\frac{\delta_k}{l_k^m}$, so the first assertion in (\ref{2.23})
holds. For the second assertion, we see directly that it holds when
$k=1$, as: $\|\nabla v_1\|_m = \frac{1}{l^m} l^m \|\nabla {\mathfrak
  v}\|_m \leq \frac{\delta_1}{l^m}$. On the other hand, using
induction on $k$ and exploiting (\ref{stephld1v}), we get:
\begin{equation*}
\begin{split}
\|\nabla v_{k+1}\|_m & \leq  \|\nabla v_k\|_m + \|\nabla v_{k+1} -
\nabla v_k\|_m \leq \frac{\delta_k}{l_k^m} + 
C \delta_{k}\lambda_k^m \\ & \leq \delta_k\Big(\frac{1}{l_{k+1}^m} +
\frac{C}{l_{k+1}^m}\Big) = C\frac{\delta_k}{l_{k+1}^m} \leq
\frac{\delta_{k+1}}{l_{k+1}^m} \quad \qquad \forall
m=1,2 \quad \forall k=1,2.
\end{split}
\end{equation*}
The proof of (\ref{2.23}) is now complete for the choice $\delta_{k+1}
= C\delta_k$, where $C>1$ is, as always, an appropriately large
universal constant. Consequently: $\delta_2, \delta_3\leq
C\|\mathcal{D}\|_0^{1/2}<1$ if only $\delta_0\ll 1$.

\smallskip

{\bf 4.} We now directly verify the concluding estimates of Proposition
\ref{stg-hld}. We have, in view of the definition of $ {\mathfrak D}'$ and (\ref{mo}):
\begin{equation*}
\begin{split}
\tilde{\mathcal{D}} & = A -  {\mathfrak A} + {\mathfrak D}' + \big(\frac{1}{2} \nabla
v_1\otimes \nabla v_1 + \sym\nabla w_1\big) - \big(\frac{1}{2} \nabla
v_4\otimes \nabla v_4 + \sym\nabla w_4\big) \\ & = A - {\mathfrak A} - \sum_{k=1}^3
\Big(\big(\frac{1}{2} \nabla
v_{k+1}\otimes \nabla v_{k+1} + \sym\nabla w_{k+1}\big) - \big(\frac{1}{2} \nabla
v_k\otimes \nabla v_k + \sym\nabla w_k + a_k\xi_k\otimes \xi_k\big)\Big),
\end{split}
\end{equation*}
and thus by (\ref{mol1}), (\ref{stephld2}) and the definition of $l$, there
follows (\ref{stg-hld1}):
\begin{equation*}
\begin{split}
\|\tilde{\mathcal{D}}\|_0 & \leq \| A -  {\mathfrak A}\|_0 + C
\sum_{k=1}^3\frac{\delta_k^2}{\lambda_k l_k} \leq C\Big(l^\beta
\|A\|_{0,\beta} + \delta_3^2\sum_{k=1}^3\frac{1}{\lambda_k l_k}\Big)
\\ & \leq C \Big(\frac{\|\mathcal D\|_0^{\beta/2}}{M^\beta} \|A\|_{0,\beta}
+ 3 \frac{\delta_3^2}{\sigma}\Big) \leq C \Big(\frac{\|\mathcal D\|_0^{\beta/2}}{M^\beta} \|A\|_{0,\beta}
+ \frac{1}{\sigma}\|\mathcal{D}\|_0\Big).
\end{split}
\end{equation*}
We now check (\ref{stg-hld2v}), using (\ref{mol1}), (\ref{res}) and (\ref{stephld1w}):
\bee\label{gou}
\begin{split}
\|\tilde v - v\|_1 & \leq \|{\mathfrak v} - v\|_1 +
\sum_{k=1}^3\|v_{k+1}- v_k\|_1 \leq C\|\mathcal{D}\|_0^{1/2} + C
\sum_{k=1}^3\delta_k \leq C\|\mathcal{D}\|_0^{1/2} \\ 
\|\tilde w - w\|_1 & \leq \|{\mathfrak w} - w\|_1 + \|{\mathfrak w}' - {\mathfrak w}\|_1 +
\sum_{k=1}^3\|w_{k+1}- w_k\|_1 \\ & \leq C\Big(\|\mathcal{D}\|_0^{1/2} +  \|\mathcal{D}\|_0
+ \sum_{k=1}^3\delta_k (1+ \|\nabla v_k\|_0)\Big) \leq C\|\mathcal{D}\|_0^{1/2} 
\Big(1+ \sum_{k=1}^3 \|\nabla v_k\|_0\Big)  \\ & \leq 
C\|\mathcal{D}\|_0^{1/2}  \Big(1+ \|\nabla v\|_0 + \|{\mathfrak v} -
v\|_1 + \sum_{k=1}^2 \| v_{k+1}-v_k\|_1\Big)  \\ & \leq
C\|\mathcal{D}\|_0^{1/2} \big(1+ \|\nabla v\|_0 + \|\mathcal{D}\|_0^{1/2} \big)  
\leq C\|\mathcal{D}\|_0^{1/2} \big(1+ \|\nabla v\|_0 \big).  
\end{split}
\eee
Finally, the first bound in (\ref{stg-hld3}) follows by
(\ref{step-v}) and (\ref{stephld1v}):
\begin{equation*}
\begin{split}
\|\tilde v\|_2 & \leq \|{\mathfrak v}\|_2 + \sum_{k=1}^3\|v_{k+1} -
v_k\|_2 \leq \frac{C}{l}\|\mathcal{D}\|_0^{1/2} + 
C \sum_{k=1}^3\delta_k\lambda_k \\ & \leq \frac{C}{l}\|\mathcal{D}\|_0^{1/2} + C\delta_3
\sum_{k=1}^3\frac{\sigma^k}{l} 
\leq \frac{C}{l}\|\mathcal{D}\|_0^{1/2} (1+\sigma^3) \leq CM\sigma^3,
\end{split}
\end{equation*}
while the second bound is obtained by:
\begin{equation*}
\begin{split}
\|\tilde w\|_2 & \leq \|{\mathfrak w}\|_2 + \|{\mathfrak w}' - {\mathfrak w}\|_2
+ \sum_{k=1}^3\|w_{k+1} - w_k\|_2 \leq C \Big(M + \|\mathcal{D}\|_0+ 
\sum_{k=1}^3\delta_k\lambda_k (1+\|\nabla v_k\|_0)\Big) \\ & \leq C
\Big(M + \delta_3\sum_{k=1}^3\frac{\sigma^3}{l} (1+\|\nabla
v_k\|_0)\Big) \leq C M \Big(1 + \sigma^3 + \sigma^3 \sum_{k=1}^3\|\nabla v_k\|_0\Big) 
\\ & \leq CM\sigma^3\Big(1+ \sum_{k=1}^3\|\nabla v_k\|_0\Big) \leq
CM\sigma^3\big(1+\|\nabla v\|_0\big).
\end{split}
\end{equation*}
in view of (\ref{step-w}), (\ref{res}) and reasoning as in (\ref{gou}).
\end{proof}
 
\section{The $\mathcal{C}^{1,\alpha}$ approximations - a proof of
  Theorem \ref{w2oi-hld}.}\label{nareszcie} 

We are now in a position to state the final intermediary approximation
result, parallel to {\cite[Theorem 1]{CDS}}.  

\begin{theorem}\label{Holderdelta_01}
Assume that $\Omega\subset\mathbb{R}^2$ is an open, bounded
domain. Given are functions $v\in\mathcal{C}^2(\bar\Omega)$, $w \in
\mathcal{C}^2(\bar \Omega, \R^2)$ and $A \in
\mathcal{C}^{0,\beta}(\bar \Omega, \R^{2\times 2}_{sym})$ for some
$\beta\in (0,1)$, such that the deficit $\mathcal{D}$ below is
appropriately small:
\begin{equation}\label{delta_0}
\mathcal{D} = A - \big(\frac 12 \nabla  v \otimes \nabla v + \sym\nabla
w\big), \qquad 0<\|\mathcal{D}\|_0<\delta_0\ll  1. 
\end{equation}
Fix the exponent:
\begin{equation}\label{expo}
0<\alpha< \min \Big\{\frac 17, \frac{\beta}2\Big\}.
\end{equation}
Then, there exist  $\bar v \in \mathcal{C}^{1,\alpha}(\bar \Omega)$
and $\bar  w\in \mathcal{C}^{1,\alpha}(\bar \Omega, \R^2)$  such that:
\bee\label{equm}
\frac 12 \nabla \bar v  \otimes \nabla \bar v + \sym \nabla \bar w = A, 
\eee 
\bee\label{estmv}
 \|\bar v - v\|_1  \le C \|\mathcal{D}\|^{1/2}_0 \qquad \mbox{ and }
\qquad \|\bar w - w \|_1  \leq C (1+ \|\nabla \tilde v\|_0) \|\mathcal{D}\|^{1/2}_0,
\eee 
where $C>0$ is a constant depending on $\alpha$ but independent of all other parameters.
\end{theorem}
\begin{proof} 
The exact solution to (\ref{equm}) will be obtained as the
$\mathcal{C}^{1,\alpha}$ limit of sequences of successive
approximations $\{v_k\in \mathcal{C}^2(\bar\Omega), ~ w_k \in
\mathcal{C}^2(\bar \Omega, \R^2)\}_{k=1}^\infty$. 

\smallskip

{\bf 1. Induction on stages.}  
We set $v_0=v$ and $w_0=w$. Given $v_k$ and $w_k$, define $v_{k+1}$ and $w_{k+1}$ by
applying Proposition \ref{stg-hld} with parameters $\sigma$ and
$M_k$ that will be appropriately chosen below and that satisfy:
\begin{equation}\label{lew}
  M_k>\max\{\|v_k\|_2, \|w_k\|_2, 1\} \qquad \mbox{ and } \qquad \sigma >1.
\end{equation}
Following our notational convention, we define the $k$-th deficit
$\mathcal{D}_k = A - \big(\frac 12 \nabla  v_k \otimes \nabla v_k +
\sym\nabla w_k \big)$. In view of Proposition  \ref{stg-hld}, we get:
\bee\label{stgs1}
\|\mathcal{D}_{k+1}\|_0 \leq C  \Big(\frac{\|A\|_{0,\beta}}{M_k^\beta} \|\mathcal{D}_k\|_0^{\beta/2}
+ \frac{1}{\sigma} \|\mathcal{D}_k\|_0\Big),
\eee
\bee\label{stgs2v}
\|v_{k+1} - v_k\|_1  \leq C \|\mathcal{D}_k\|_0^{1/2} \qquad \mbox{
  and }\qquad
\|w_{k+1} - w_k\|_1 \leq C (1+ \|\nabla v_n\|_0) \|\mathcal{D}_k\|_0^{1/2},
\eee
\bee\label{stgs3}
\|v_{k+1} \|_2   \leq C   M_k \sigma^3  \qquad  \mbox{ and } \qquad 
\|w_{k+1}\|_2 \le  C(1+ \|\nabla v_k\|_0) M_k \sigma^3,
\eee
provided that (\ref{assumpkot}) holds for each $\mathcal{D}_k$. We
shall now validate this requirement, with the parameters:
\begin{equation}\label{Mn}
M_k = \Big({\mathfrak C} (1+\|\nabla v_0\|_0) \sigma^3\Big)^k M_0.
\end{equation}
In fact, we will inductively prove that one can have:
\bee\label{indstep} 
\|\mathcal{D}_k\|_0  \leq  \frac{1}{\sigma^{sk}} \|\mathcal{D}\|_0   
\qquad \mbox{ with any } \qquad 0<s<\min\big\{1, \frac{6\beta}{2-\beta}\big\}.
\eee 

Fix $s$ as indicated in (\ref{indstep}). Clearly, (\ref{indstep}) and
(\ref{lew}) hold for $k=0$. By (\ref{stgs1}) and the induction
assumption we obtain the bound:
\begin{equation}\label{ges} 
\sigma^{s(k+1)}\frac{\|\mathcal{D}_{k+1}\|_0}{\|\mathcal{D}\|_0} \leq
\frac{C \|A\|_{0,\beta} \|\mathcal{D}\|_0^{\beta/2 -1}
  \sigma^s}{M_0^\beta} \frac{1}{{\mathfrak C}^{k\beta}}
\Bigg(\frac{\sigma^{(1-\beta/2)(s-\frac{6\beta}{2-\beta})}}{(1+\|\nabla
  v_0\|_0)^\beta}\Bigg)^k + C\sigma^{s-1}.
\end{equation}
We see that in view of the condition on $s$ in (\ref{indstep}), both
$\sigma^{s-1}$ and $\sigma^{(1-\beta/2)(s-\frac{6\beta}{2-\beta})}$
are smaller than $1$. Further, it is possible to choose $\sigma>1$ so
that the second term in (\ref{ges}) be smaller than $1/2$ and that the
quotient term in parentheses above is also smaller than $1$. Then,
choose $M_0$ so that (\ref{lew}) holds for $k=0$ together with:
$$\frac{C \|A\|_{0,\beta} \|\mathcal{D}\|_0^{\beta/2 -1}
  \sigma^s}{M_0^\beta} <\frac{1}{2}.$$
This results in the first term in (\ref{ges}) being smaller than $1/2$
if only ${\mathfrak C}\geq 1$. Consequently, we get
$\sigma^{s(k+1)}{\|\mathcal{D}_{k+1}\|_0}/{\|\mathcal{D}\|_0} \leq 1$
as needed in (\ref{indstep}).

Observe now that by (\ref{stgs2v}) and by the established (\ref{indstep}):
\bee\label{kruk}
\begin{split}
\forall k\geq 0 \qquad \|\nabla v_{k} \|_0 & \leq \|\nabla v_0\|_0 +   
\sum_{i=0}^{k-1} \|v_{i+1} - v_i\|_1 \leq  \|\nabla v_0\|_0 + C \sum_{i=0}^{k-1}
\|\mathcal{D}_i\|_0^{1/2} \\ & \leq  \|\nabla v_0\|_0 +  C \Big(\sum_{i=0}^\infty
\frac{1}{\sigma^{si/2}}\Big)\|\mathcal{D}\|_0^{1/2} = 
\|\nabla v_0\|_0 +  \frac{C}{1-\sigma^{-s/2}}\|\mathcal{D}\|_0^{1/2}
\\ & \leq \|\nabla v_0\|_0 +  C \|\mathcal{D}\|_0^{1/2},
\end{split}
\eee
if only, say, $\sigma^s>4$ which can be easily achieved through the
choice of $\sigma$. Now, by (\ref{stgs3}) and (\ref{kruk}):
\begin{equation*}
\frac{\|v_{k+1}\|_2}{M_{k+1}} \leq \frac{1}{{\mathfrak C}}
\frac{C}{(1+\|\nabla v_0\|_0)} \quad \mbox{ and } \quad
\frac{\|w_{k+1}\|_2}{M_{k+1}} \leq \frac{1}{{\mathfrak C}}
\frac{C (1+\|\nabla v_k\|_0)}{(1+\|\nabla v_0\|_0)}\leq \frac{1}{{\mathfrak C}}
\frac{C ( 1+\|\nabla v_0\|_0+ \|\mathcal{D}\|_0^{1/2})}{(1+\|\nabla v_0\|_0)}.
\end{equation*}
Hence, taking the constant ${\mathfrak C}\gg 1$ large enough, we see that both
quantities above can be made smaller than $1$, proving therefore the
required (\ref{lew}).

\smallskip
  
{\bf 2. $\mathcal{C}^{1, \alpha}$ control of the approximating sequences $v_n$ and $w_n$.}
Let now $\alpha$ be an exponent as in (\ref{expo}). Choose $s$
satisfying (\ref{indstep}) and:
\begin{equation}\label{rob}
\alpha (6+s) - s <0.
\end{equation}
It is an easy calculation that $s$ satisfying (\ref{indstep}) and
(\ref{rob}) exists if and only if the exponent $\alpha$ is in the
range (\ref{expo}). Indeed, (\ref{rob}) is equivalent to
$\alpha<\frac{s}{6+s}$, while (\ref{indstep}) is equivalent to:
$$0<\frac{s}{6+s}<\min\Big\{\frac{1}{7}, \frac{\beta}{2}\Big\}.$$

We will prove that sequences $\{v_k, w_k\}_{k=0}^\infty$ are Cauchy in
$\mathcal{C}^{1,\alpha}(\bar\Omega)$. Firstly, by \eqref{stgs2v},
\eqref{kruk},\eqref{indstep}:
\bee\label{gli}
\begin{split}
\|v_{k+1} - v_{k}\|_1 & \leq C \|\mathcal{D}_k\|_0^{1/2}\leq \frac{C}{\sigma^{sk/2}}\|\mathcal{D}\|_0^{1/2}, \\ 
\|w_{k+1} - w_k\|_1 & \leq C (1+ \|\nabla v_k\|_0)
\|\mathcal{D}_k\|_0^{1/2} \leq \frac{C}{\sigma^{sk/2}}\big( 1+ \|\nabla
  v_0\|_0 + \|\mathcal{D}\|_0^{1/2}\big)\|\mathcal{D}\|_0^{1/2},
\end{split} 
\eee
so we see right away that they are Cauchy in
$\mathcal{C}^1(\bar\Omega)$. On the other hand, by \eqref{stgs3},
\eqref{kruk}, \eqref{indstep}:
\begin{equation*}
\|v_{k+1} - v_{k}\|_2  + \|w_{k+1}- w_k\|_2  \leq  C(1+ \|\nabla v_k\|_0) M_k \sigma^3 \leq C 
\big( 1+ \|\nabla  v_0\|_0 +
\|\mathcal{D}\|_0^{1/2}\big)\Big({\mathfrak C} (1+\|\nabla v_0\|_0)\sigma^{3}\Big)^k M_0,
\end{equation*}
so the sequences have the tendency to diverge in $\mathcal{C}^2(\bar\Omega)$.
Interpolating now the $\mathcal{C}^{1,\alpha}$ norm by \cite{Adams}:
$$ \|f\|_{0,\alpha} \le \|f\|_1^\alpha \|f\|^{1-\alpha}_0,$$ 
we obtain:
\bee\label{wro}
\begin{split}
\|\nabla(v_{k+1}  - v_k)\|_{0,\alpha}   + \|\nabla(w_{k+1} -
w_k)\|_{0,\alpha} & \leq C_0^{\alpha} (C_0\sigma^3)^{k\alpha}
M_0^\alpha \cdot C_0^{1-\alpha} \frac{1}{\sigma^{sk(1-\alpha)/2}} \\ &
= C_0 M_0^\alpha (C_0^\alpha)^h \Big( \sigma^{\frac{\alpha(6+s)-s}{2}}\Big)^k,
\end{split} 
\eee
where by $C_0$ we denoted an upper bound of all quantities involving
$C$, $v_0$, $\mathcal{D}$. It is clear that choosing $\sigma$
sufficiently large (so that $C_0 \sigma^{3-s/2}<1$), the resulting
bound (\ref{wro}) implies that $\{\nabla v_k, \nabla
w_k\}_{k=0}^\infty$  are Cauchy in $\mathcal{C}^{0,
  \alpha}(\bar\Omega)$, provided that (\ref{rob}) holds. We see that
the choice of exponent range in (\ref{expo}) so that the above
construction technique works, is optimal.

\smallskip

{\bf 3.} Concluding, we see that $\{v_k, w_k\}_{k=0}^\infty$  converge
to some $\bar v\in \mathcal{C}^{1,\alpha}(\bar \Omega)$ and $\bar w \in
\mathcal{C}^{1,\alpha} (\bar \Omega, \R^2)$. Since the defects in the
approximating sequence obeys: $\lim_{k\to\infty}\|\mathcal{D}_k\|_0=0$
by (\ref{indstep}), we immediately get (\ref{equm}). Additionally, by (\ref{gli}):
\begin{equation*}
\begin{split}
\|\bar v -  v\|_1 & \leq \sum_{k=0}^\infty \|v_{k+1}
- v_k\|_1 \leq C \Big(\sum_{k=0}^\infty \frac{1}{\sigma^{sk/2}}\Big)
\|\mathcal{D}\|_0^{1/2} = \frac{C}
{1-\sigma^{-s/2}}\|\mathcal{D}\|_0^{1/2}\leq C \|\mathcal{D}\|_0^{1/2}\\
\|\bar w - w\|_1 & \leq \sum_{k=0}^\infty \|w_{k+1}
- w_k\|_1 \leq C \Big(\sum_{k=0}^\infty \frac{1}{\sigma^{sk/2}}\Big)
(1+ \|\nabla v\|_0) \|\mathcal{D}\|_0^{1/2} \leq 
C (1+ \|\nabla v\|_0)  \|\mathcal{D}\|_0^{1/2}.
\end{split}
\end{equation*}
completing the proof of (\ref{estmv}).
\end{proof}
 
\bigskip

We are now ready to give:

\bigskip

\noindent {\bf Proof of Theorem \ref{w2oi-hld}.}

\noindent Fix a sufficiently small $\varepsilon>0$. We will construct
$\bar v \in \mathcal{C}^{1,\alpha}(\bar\Omega)$ and $\bar
w\in\mathcal{C}^{1,\alpha}(\Omega, \R^2)$ such that:
\begin{equation}\label{10}
A_0 = \frac 12 \nabla \bar v\otimes \nabla \bar v +\sym\nabla \bar w
\qquad \mbox{ in } \bar\Omega
\end{equation}
and:
\begin{equation}\label{11}
\|\bar v - v_0\|_0 + \|\bar w - w_0\|_0 < \varepsilon.
\end{equation}
In order to apply Theorem \ref{Holderdelta_01}, we need to decrease
the deficit $A_0 - \big(\frac 12 \nabla v_0\otimes \nabla v_0 +\sym\nabla w_0\big)$
so that it obeys (\ref{delta_0}). This will be done in three
steps. 

First, let $ \tilde v_0 \in \mathcal{C}^\infty(\bar\Omega)$,
$\tilde w_0\in\mathcal{C}^\infty(\Omega, \R^2)$ and $\tilde
A_0\in\mathcal{C}^\infty(\bar\Omega, \R^{2\times 2}_{sym})$ be such that:
\begin{equation}\label{12}
\begin{split}
& \|\tilde v_0 - v_0\|_1 + \|\tilde w_0 - w_0\|_1 + \|\tilde A_0 - A_0\|_0 < \varepsilon^2\\
&\exists \tilde c_0>0 \qquad A_0 - \big( \frac 12 \nabla \tilde
v_0\otimes \nabla \tilde v_0 +\sym\nabla \tilde w_0 ) > \tilde c_0
\mbox{Id}_2 \qquad \mbox{ in } \bar\Omega.
\end{split}
\end{equation}
Second, by Theorem \ref{weak2oi} and Remark \ref{remi}, there exists $v \in \mathcal{C}^{1}(\bar\Omega)$ and 
$w\in\mathcal{C}^{1}(\Omega, \R^2)$ such that: 
\begin{equation}\label{13}
\begin{split}
& \tilde A_0 = \frac 12 \nabla  v\otimes \nabla v +\sym\nabla w \qquad \mbox{ in } \bar\Omega,\\
& \|v - \tilde v_0\|_0 + \|w - \tilde w_0\|_0 < \varepsilon^2 \qquad
\mbox{ and } \qquad \|\nabla v - \nabla\tilde v_0\|_0 \leq C.
\end{split}
\end{equation}
Third, let $\tilde v \in \mathcal{C}^{2}(\bar\Omega)$ and $\tilde
w\in\mathcal{C}^{2}(\Omega, \R^2)$ be such that:
\begin{equation}\label{14}
\|v - \tilde v\|_1 + \|w - \tilde w\|_1 < \varepsilon^2. 
\end{equation}

By (\ref{13}), (\ref{14}) and (\ref{12}), we get:
\begin{equation}\label{15}
\begin{split}
\|A_0 - \big( \frac 12 \nabla  &\tilde v\otimes \nabla \tilde v
+\sym\nabla \tilde w \big)\|_0 \\ & \leq \|A_0 - \tilde A_0\|_0 + 
\|\big( \frac 12 \nabla  \tilde v\otimes \nabla \tilde v
+ \sym\nabla \tilde w \big) - \big( \frac 12 \nabla v\otimes \nabla v
+ \sym\nabla w \big)\|_0 \\ & \leq  \|A_0 - \tilde A_0\|_0 + \big(\|\nabla
v\|_0 + \|\nabla \tilde v\|_0\big)\|\nabla v - \nabla\tilde v\|_0 +
\|\nabla w- \nabla\tilde w\|_0 \\ & \leq \varepsilon^2 + \big(2\|\nabla
v_0\|_0 + 2\varepsilon^2 + C\big)\varepsilon^2 +\varepsilon^2 <\delta_0,
\end{split}
\end{equation}
as required in Theorem \ref{Holderdelta_01}, if only $\varepsilon$ is small enough.
We now apply Theorem \ref{Holderdelta_01} to $\tilde v$, $\tilde w$
and the original field $A_0$, and get $\bar v \in \mathcal{C}^{1,\alpha}(\bar\Omega)$ and $\bar
w\in\mathcal{C}^{1,\alpha}(\Omega, \R^2)$ satisfying (\ref{10}) and
such that:
\begin{equation*}
\begin{split}
\|\bar v - v_0 \|_0 + \|\bar w - w_0 \|_0 & \leq
C \big(1 + \|\nabla \tilde v\|_0\big) \|A_0 - \big( \frac 12 \nabla  \tilde v\otimes \nabla \tilde v
+\sym\nabla \tilde w \big)\|_0 + 3\varepsilon^2 \\ & \leq C \big(1+\varepsilon^2
+ \|\nabla v_0\|_0\big)^2 \varepsilon^2 + 3\varepsilon^2,
\end{split}
\end{equation*}
by (\ref{estmv}), (\ref{15}), (\ref{14}), (\ref{13}) and
(\ref{12}). Clearly (\ref{11}) follows, if $\varepsilon$ is small enough. 
\endproof
  

\bigskip

The following  Corollary is of independent interest: 

\begin{corollary}\label{ss}
Let $\Omega, f, p, \alpha$ be as in the statement of Theorem \ref{weakMA}. Let $q\geq 2$.
Then, for all $v_0\in W^{1,q}(\Omega)$, there exists a sequence
$v_n\in \mathcal{C}^{1,\alpha}(\bar\Omega)$ weakly converging to $v_0$ in $W^{1,q}(\Omega)$, and such that:
$\Det \nabla^2 v_n = f$ in $\Omega$.
\end{corollary}

\begin{proof}
Let  $\bar v_n\in\mathcal{C}^1(\bar\Omega)$ converge to $v_0$ in
$W^{1,q}(\Omega)$. For every $\bar v_n$, consider the approximating sequence
$\{v_{n,k}\in\mathcal{C}^{1,\alpha}(\bar\Omega)\}_{k=1}^\infty$ as in
Theorem \ref{weakMA}, converging uniformly to $\bar v_n$. Define now
$\{v_n\}$ to be an appropriate diagonal sequence, so that it converges
to $v_0$ in $L^q(\Omega)$. We will check that $\{v_n\}$ is bounded in
$W^{1,q}$.

The boundedness of $\|v_n\|_{L^q}$ is clear from the convergence
statement. On the other hand, the proof of Theorem \ref{w2oi-hld} gives,
by (\ref{estmv}), (\ref{12}), (\ref{13}), (\ref{14}) and (\ref{15}):
$$|\nabla v_n(x)| \leq |\nabla \bar v_n(x)| + 2\epsilon^2 + C +
C\delta_0^{1/2} \leq |\nabla \bar v_n(x)| + C \qquad \forall x\in \Omega.$$
Consequently, $\|\nabla v_n\|_{L^q} \leq \|\nabla\bar
v_n\|_{L^q} + C\leq C$, which concludes the proof.
\end{proof}

\section{Rigidity results for $\alpha>2/3$ - a proof of Theorem
  \ref{rig1}.}
 
The crucial element in the proof of the rigidity Theorems \ref{rig1}
and \ref{rig2} is the following result, that is the \lq small slope
analogue' of \cite[Proposition 6]{CDS}: 

\begin{proposition}\label{deg}
Let $\Omega \subset \R^2$  be an open, bounded, simply connected domain.
Assume that for some
$\alpha\in (2/3, 1)$,  the function $v\in \mathcal{C}^{1,\alpha}(\bar \Omega)$ is a solution to: 
$$ \Det \nabla^2 v = f \qquad \mbox{ in } \bar\Omega, $$ 
where $f\in L^{p}(\Omega)$ and $p>1$.
Then the following degree formula holds true, for every open subset
$U$ compactly contained in $\Omega$ and every $g\in
L^\infty(\mathbb{R}^2)$ with $\mathrm{supp } ~g\subset \R^2\setminus \nabla v(\partial U)$:
\bee\label {area}
\int_U  (g \circ \nabla v)f  = \int_{\R^2} g(y) \deg(\nabla v, U, y)~ \mathrm{d}y. 
\eee
Above,  $\deg(\psi, U, y)$ denotes the Brouwer degree of a continuous
function $\psi:\bar U\to \R^2$ at a point $y\in \R^2 \setminus \psi(\partial U)$.  
\end{proposition}

\begin{proof}
{\bf 1.} Fix $U$ and $g$ as in the statement of the Proposition. We first
recall \cite{Lloyd} that $\deg(\nabla v, U, \cdot)$  is well defined on the open
set $\R^2 \setminus \nabla v(\partial U)$. In fact, this function is
constant on each connected component $\{U_i\}_{i=0}^\infty$ of $\R^2
\setminus \nabla v(\partial U)$ and it equals $0$ on the only
unbounded component $U_0=\R^2\setminus\nabla v(\bar U)$. Thus, without
loss of generality, we may assume that $g$ is compactly supported and
that: $\mbox{supp } g\subset \bigcup_{k=1}^\infty U_k$. By compactness, there
must be: $\mbox{supp } g\subset \bigcup_{k=1}^N U_k$ for some $N$, and
consequently the integral in the right hand side of \eqref{area} is
well defined.

Let now $\{g_i\in\mathcal{C}_c^\infty(\bigcup_{k=1}^N
U_k)\}_{i=1}^\infty$ be a sequence pointwise converging to $g$ and
such that $\|g_i\|_0\leq \|g\|_{L^\infty}$ for all $i$. It is
sufficient to prove the formula
\eqref{area} for each $g_i$ and pass to the limit by  dominated convergence
theorem. To simplify the notation, we drop the index, and so in what
follows we are assuming that $g\in \mathcal{C}^\infty_c \big((\R^2\setminus \nabla
v(\partial U))\cap \nabla v(\bar U)\big)$. 

As in the proof of Theorem \ref{weakMA}, let $A\in W^{2,p}(\Omega)\cap
\mathcal{C}^{0,\beta}(\bar\Omega)$ be such that $\mbox{curl}\mbox{curl} A = -f$. 
Here, we take $\beta=\min\{2-\frac{2}{p}, \alpha\}\in (0,1)$.
Consequently, in view of the simple connectedness of $\Omega$, there exists
  $w\in\mathcal{C}^{1,\beta}(\bar\Omega,\R^2)$ such that:
$$A = \frac 12  \nabla v\otimes \nabla v + \sym \nabla w. $$

For a standard 2d mollifier $\varphi \in \mathcal{C}_c^\infty(B(0,1))$
as in Lemma \ref{stima}, define:
$$ \forall l\in (0,1) \qquad v_{l} = v\ast \varphi_l, \quad w_{l} = w\ast \varphi_l, \quad
A_l= A\ast \varphi_l, $$ 
and apply the area formula (change of variable formula \cite{EG,
  AFP}) to the smooth functions $g$ and $\nabla v_l$:
\begin{equation}\label{ptak}
\int_U  (g \circ \nabla v_l) \det \nabla^2 v_l  = \int_{\R^2} g(y)
\deg(\nabla v_l, U, y)~\mbox{d}y.
\end{equation}
We see that $\nabla v_l$ converge uniformly to $\nabla v$, so 
the degrees converge pointwise \cite{Lloyd} and thus:
$$ \lim_{l\to 0} \int_{\R^2} g(y) \deg(\nabla v_l, U, y)~\mbox{d}y =
\int_{\R^2} g(y) \deg(\nabla v, U, y)~\mbox{d}y.$$
To conclude the proof in view of (\ref{ptak}), it suffices now to show that:
\bee\label{want}
\lim_{l\to 0} \int_U  (g \circ \nabla v_l) \det \nabla^2 v_l   = \int_U (g \circ \nabla v) f.
\eee 

\smallskip

{\bf 2.}
Following \cite{CDS, CoTi} we use a commutator estimate to
get (\ref{want}). As $f=-\mbox{curl }\mbox{curl} A$, we have:
\bee\label{zaza}
\begin{split}
\Big  |\int_U (g \circ \nabla v_l) &\det \nabla^2 v_l - (g
\circ \nabla v) f \Big |  \\  &
\leq  \Big | \int_U (g \circ \nabla v_l) \Big(\det \nabla^2 v_l + \mbox{curl }\mbox{curl} A_l
\Big)  \Big | + \Big | \int_U (g \circ \nabla v_l) \mbox{curl
}\mbox{curl} (A_l - A)  \Big | \\ & \qquad
+ \Big |  \int_U  \Big( (g \circ\nabla v_l)  -  (g \circ \nabla v)\Big) f \Big |.
\end{split}  
\eee
The second term above is bounded by $C\int_U |\nabla^2A_l - \nabla^2 A|
\leq C\|A_l - A\|_{W^{2,p}(\Omega)}$, hence it converges to $0$. The
third term also converges to $0$ by the dominated convergence theorem,
since $g\circ \nabla v_l$ converges to $g\circ \nabla v$. In order
to deal with the first term in (\ref{zaza}), observe that
$\det\nabla^2v_l = -  \mbox{curl }\mbox{curl} \big(\frac 12 \nabla
v_l\otimes\nabla v_l + \sym\nabla w_l\big)$ and integrate by parts, in
view of $g\circ \nabla v_l=0$ on $\partial U$:
\bee\label{zaza1}
\begin{split}
\Big  |\int_U (g \circ \nabla v_l) \big(&\det \nabla^2 v_l + \mbox{curl }\mbox{curl} A_l\big)
\Big |  \\ & =  \Big  |\int_U \Big\langle\nabla^\perp (g \circ \nabla
v_l), \mbox{curl} \big( \frac 12 \nabla v_l \otimes \nabla
 v_l + \sym \nabla w_l - A_l\big)\Big\rangle  \Big |  \\ & \leq
C \|\nabla g\|_0 \|\nabla^2v_l\|_0 \big\| \nabla v_l \otimes \nabla v_l  -  (\nabla v
\otimes \nabla v) \ast \varphi_l\big\|_1  \\ &
\leq C \frac{1}{l^{1-\alpha}}\|\nabla v\|_{0,\alpha}\cdot
\frac{1}{l^{1-2\alpha}} \|\nabla v\|_{0,\alpha}^2 = C \frac{1}{l^{2-3\alpha}}\|\nabla v\|_{0,\alpha}^3,
\end{split}   
\eee 
where we used  Lemma \ref{stima}.
Clearly, for $\alpha>2/3$ the right hand side in (\ref{zaza1})
converges to $0$ as $l\to 0$. By (\ref{zaza}), this implies
(\ref{want}) and concludes the proof.
\end{proof} 

\bigskip

Below, we present all the details of the proof of Theorem
\ref{rig1}. The proof of Theorem \ref{rig2} will be postponed to \cite{LPconvex}.

\bigskip

\noindent {\bf Proof of Theorem \ref{rig1}.}

\noindent 
{\bf 1.} By Proposition \ref{deg} it follows that for all  open sets
$U\subset\bar U\subset\Omega$:
\begin{equation}\label{kotek1}
\deg (\nabla v, U, y)=0 \qquad \forall y\in\mathbb{R}^2\setminus
\nabla v(\partial U). 
\end{equation}
We would like to conclude \cite{Po56, Po73} that the image set
$\nabla v(U)$ is of measure $0$. This will result in the
developability of $v$, by the main statement of \cite{Kor}.
However, we note that there exist a H\"older continuous vector field
whose local degree vanishes everywhere, but whose image is onto the
unit square \cite{LM-private}. Therefore, we will additionally exploit
the gradient structure of $\nabla v$, using ideas of  \cite[Chapter
2]{kirchthesis}, in combination with the commutator estimate
technique as in the proof of Proposition \ref{deg}.

Let $v_l = v\ast\varphi_l$ be as in the proof of Proposition \ref{deg}
and for every $\delta>0$ define:
$$ u_{l, \delta}(x_1, x_2) = \nabla v_l (x_1, x_2)  + \delta (-x_2,
x_1), \qquad u_{\delta}(x_1, x_2) = \nabla v (x_1, x_2)  + \delta (-x_2, x_1). $$  
Fix an open set $U$ with smooth boundary and compactly contained in
$\Omega$. Let $g\in\mathcal{C}_c^\infty\Big( (\R^2\setminus\nabla v(\partial
U))\cap\nabla v(\bar U)\Big)$, and use the change of variable formula
to $g$ and $u_{l,\delta}$:
\bee\label{deg-del1}
\int_U  (g \circ u_{l, \delta}) \big(\det \nabla^2 v_l + \delta^2\big)
= \int_{\R^2} g(y) \deg(u_{l, \delta}, U, y)~\mbox{d}y,
\eee
where we noted that $\det \nabla u_{l, \delta} = \det \nabla^2 v_l + \delta^2$.  
The integral in the right hand side of (\ref{deg-del1}) is well
defined for sufficiently small $l$ and $\delta$, because then
$y\in\mbox{supp } g$ implies $y\not\in u_{l,\delta}(\partial U)$.

Passing to the limit, we immediately obtain:
\bee\label{deg-del2}
\lim_{l\to 0} \int_{\R^2} g(y) \deg(u_{l, \delta}, U, y)~\mbox{d}y =
\int_{\R^2} g(y) \deg(u_{\delta}, U, y)~\mbox{d}y,
\eee 
while the left hand side of (\ref{deg-del1}) can be estimated by:
\bees
\Big  |\int_U  (g \circ u_{l, \delta}) \big(\det \nabla^2 v_l + \delta^2\big) - (g
\circ u_\delta) \delta^2 \Big | 
\leq  \Big  |\int_U  (g \circ u_{l, \delta}) \det \nabla^2 v_l \Big | 
+ \Big | \int_U  (g \circ u_{l, \delta} - g \circ u_\delta) \delta^2 \Big |.
\eees
The second term above clearly converges to $0$ as $l\to 0$, because
$u_{l,\delta}$ converge to $u_\delta$. The first term also converges
to $0$ as $\alpha >2/3$, where we reason exactly as in (\ref{zaza})
and (\ref{zaza1}), keeping in mind that $f=0$. We hence conclude:
$$ \lim_{l\to 0} \int_U  (g \circ u_{l, \delta}) \big(\det \nabla^2
v_l + \delta^2\big)  = \int_U (g \circ u_\delta) \delta^2. $$ 
In view of (\ref{deg-del2}) and (\ref{deg-del1}) this implies:
$$ \forall 0<\delta\ll 1\qquad \int_U  (g \circ u_{\delta}) \delta^2
= \int_{\R^2} g(y) \deg(u_{\delta}, U, y)~\mbox{d}y. $$
Consequently:
\begin{equation}\label{kotek}
\forall 0<\delta\ll 1 \quad \forall y\in u_\delta(U) \setminus  u_\delta(\partial U)
\qquad \deg(u_\delta, U, y) \geq 1.
\end{equation}

\smallskip

{\bf 2.} We now claim that:
\begin{equation}\label{mysz}
\nabla v(U) \subset \nabla v(\partial U).
\end{equation} 
To prove (\ref{mysz}) we argue by contradiction, assuming that for some $x_0\in U$ there
is: $y_0= \nabla v(x_0) \in \nabla v(U) \setminus  \nabla v (\partial U)$. 
We distinguish two cases:
\begin{itemize}
\item[(i)] There exist sequences $\{x_k \in U\}_{k=1}^\infty$ and
  $\delta_k\to 0^+$ as $k\to\infty$ such that $y_0 =
  u_{\delta_k}(x_k)$ for all $k$. Note that for large $k$, there must
  be $y_0\not\in u_{\delta_k} (\partial U)$ because $u_\delta$
  converges to $\nabla v$ as $\delta\to 0$ and $y_0\not\in \nabla
  v(\partial U)$. In view of (\ref{kotek}) we get: $\deg(u_{\delta_k},
  U, y_0)\geq 1$ contradicting (\ref{kotek1}).

\item[(ii)]
For all $\delta$ small enough, $y_0 \not\in u_\delta (U)$. Thus $\deg
(u_\delta, U, y_0) =0$, so by (\ref{kotek}) we get $y_0\in u_\delta
(\partial U) \subset\nabla v(\partial U) + CB(0, \delta)$,
contradicting that $y_0\not\in\nabla v(\partial U)$.
\end{itemize}

Our claim (\ref{mysz}) is now established. Since the set $\nabla
v(\partial U)$ is the image of a Hausdorff $1$d set $\partial U$ under
a $\mathcal{C}^{0,1/2}$ deformation $\nabla v$, it has Lebesgue
measure $0$ (see \cite[Lemma 4]{CDS}). Thus $\nabla v(U)$
must have measure $0$ for every smooth $U$ compactly contained in $\Omega$. The
same then must be true for the entire set $\Omega$, i.e.: $|\nabla
v(\Omega)| = 0$ and we consequently obtain:
\begin{equation}\label{kori}
\mbox{Int}\big(\nabla v(\Omega)\big) = \emptyset.
\end{equation}

\smallskip

{\bf 3.} By \cite[Corollary 1.1.2.]{Kor2}, condition (\ref{kori})
implies that every point $y\in\Omega$ has a convex open
neighbourhood $\Omega_{y}$ such that for every point
$x\in\Omega_{y}$ there is a line $L_x$ passing through $x$ so that
$\nabla v$ is constant on $L_x\cap\Omega_{y}$. The same result in the
present dimensionality has been first established in \cite{Kor}, see
also footnote on pg. 875 in \cite{Kor2} for an explanation.

We now prove that $v$ is developable. Fix $x_0\in\Omega$ and let
$[y,z]\subset \bar\Omega$ be the maximal segment passing through $x_0$
on which $\nabla v = \nabla v(x_0)$ is constant. Assume that $[y,z]$ does not extend to the boundary
$\partial\Omega$, i.e. $y\in\Omega$. We will prove that then
$\nabla v$ must be constant in an open neighbourhood of $x_0$. In
fact, we will show that:
\begin{equation}\label{setV}
V = \mbox{Int} \Big((\nabla v)^{-1} \big(\nabla v(x_0)\big)\Big)
\supset (y, z).
\end{equation}
Let $(p,q) = L_y\cap \Omega_y$. By the maximality of $[y, z]$, the segment 
$(p,q)$ is not an extension of (is not parallel to)  $[y,
z]$. Also, $\nabla v=\nabla v(x_0)$ on $(p,q)$. Take any $y_1\in
(y,z)\cap \Omega_y$ and define the open triangle
$T=\mbox{Int} \big(\mbox{span}\{p,q,y_1\}\big)$. It is easy to notice that every line
passing through any point $x\in T$ must intersect at least one of the
segments $(p,q)$ or $(y,y_1)$. Since $T\subset \Omega_y$, it follows
that  $\nabla v(x) = \nabla v(x_0)$. Hence:
$$(y, y_1)\subset T\subset V$$
and, in particular, the set $V$ in (\ref{setV}) is nonempty.

To prove (\ref{setV}) assume, by contradiction, that there exists
$y_2\in [y_1, z)$ so that:
\begin{equation}\label{conio}
(y, y_2)\subset V \quad \mbox{ but } \quad (y, y_3)\not\subset V \quad
\forall y_3\in (y_2, z).
\end{equation}
Now, the intersection $\Omega_{y_2}\cap V$ contains an open arc $C$
around the point $ (y, y_2)\cap\Omega_{y_2}$. As above, we argue that
every point in a sufficienty small open neighbourhood of the segment
$I=(y, z)\cap \Omega_{y_2}$ must have the property that every line
passing through it intersects $C$ or $I$, where $\nabla v = \nabla
v(x_0)$. Consequently $I\subset V$, contradicting (\ref{conio}) and
establishing (\ref{setV}).
\endproof

\end{document}